\documentclass[11pt]{article}
\usepackage{geometry}
\usepackage{amsmath, amssymb, amsthm}
\usepackage{epsfig}
\usepackage{graphics}
\usepackage{graphicx}
\usepackage{hyperref}
\usepackage{accents}
\usepackage{epstopdf}
\usepackage{multirow}
\usepackage{subfigure}
\usepackage{color}
\usepackage{subfigure}
\usepackage[utf8]{inputenc}
\usepackage{latexsym}
\newtheorem{theorem}{Theorem}[section]
\usepackage{cite}
\usepackage{authblk}

\begin{document}

{\LARGE \bf  
\begin{center}
   A two-timescale model of plankton-oxygen dynamics predicts formation of Oxygen Minimum Zones and global anoxia

  \end{center}
}

%The study of Oxygen Minimum Zones by exploiting the plankton-oxygen dynamics.

% Exploring Oxygen Minimum Zones with the Plankton-Oxygen Model: Understanding the Impact on Marine Ecosystems

	\vspace*{1cm}
	
\centerline{\bf Pranali Roy Chowdhury$^a$, Malay Banerjee$^{a}$, Sergei Petrovskii$^{b,}$\footnote{Corresponding author. Email: sp237@leicester.ac.uk}}

\vspace{0.5cm}
	
\centerline{ $^a$Department of Mathematics and Statistics, Indian Institute of Technology Kanpur, Kanpur, India}
		
\centerline{ $^b$ School of Computing and Mathematical Sciences, University of Leicester, Leicester LE1 7RH, UK}
	
%\centerline{ $^3$ Peoples Friendship University of Russia (RUDN University), 6 Miklukho-Maklaya St,} 
%\centerline{ Moscow 117198, Russian Federation}
		
\vspace{1cm}

\begin{center}
{\bf Abstract}
\end{center}

Decline of the dissolved oxygen in the ocean is a growing concern, as it may eventually lead to global anoxia, an elevated mortality of marine fauna and even a mass extinction. Deoxygenation of the ocean often results in the formation of Oxygen Minimum Zones (OMZ): large domains where the abundance of oxygen is much lower than that in the surrounding ocean environment. Factors and processes resulting in the OMZ formation remain controversial. We consider a conceptual model of coupled plankton-oxygen dynamics that, apart from the plankton growth and the oxygen production by phytoplankton, also accounts for the difference in the timescales for phyto- and zooplankton (making it a “slow-fast system”) and for the implicit effect of upper trophic levels. The model is investigated using a combination of analytical techniques and numerical simulations. We show that the system does not allow for persistent relaxation oscillations; instead, the blowup of the canard cycle results in plankton extinction and oxygen depletion. For the spatially explicit model, an initial non-uniform perturbation can lead to the formation of an OMZ, which then grows in size and spreads over space. For a sufficiently large timescale separation, the spread of the OMZ can result in global anoxia. 

\vspace{1.0cm}
	
\noindent
{\bf Keywords:} slow-fast dynamics; pattern formation; extinction; ocean anoxia; transients

\vspace{1cm}
\newpage 
	
%\vspace*{0mm}
\section{Introduction}

Plankton is a vital element in the complex marine food webs and biochemical cycles. Phytoplankton is the primary producer standing at the base of the marine food web. As a by-product of the primary production, phytoplankton produces oxygen in the process called photosynthesis. The produced oxygen is then used by marine fauna, e.g.~zooplankton and fish. The level of dissolved oxygen is a crucial indicator of the marine ecosystem health, as its depletion may lead to a mass mortality of aquatic species \cite{Heinze21,Watson17,Wignall96}. 
Furthermore, a considerable part of the oxygen produced in the ocean goes to the atmosphere through the ocean surface. It is estimated that around 50-80\% of atmospheric oxygen originates in the ocean. Phytoplankton therefore plays a critical role in producing and maintaining oxygen levels needed for survival both of aquatic and terrestrial species~\cite{Petsch,Berner}. 

Over the last few decades, the level of the dissolved oxygen in the ocean has shown a trend to decrease \cite{Schmidtko17,Breitburg18,Keeling,Oschlies18,Oschlies17}. This is believed to be a result of the global warming, in particular because warmer water contains less dissolved oxygen \cite{Helm11,Schmidtko17}. Also, the warming leads to a stronger stratification of the upper ocean, which reduces the O$_2$ fluxes through the ocean surface~\cite{Matear00,Bopp02}. 
Apart from the above purely physical mechanisms, there can be more subtle effects of the global warming driven by a biological feedback. An increase in the water temperature may slow down the oxygen production by phytoplankton \cite{Jones77,Li84,Robinson00,Hancke04}, potentially resulting in a regime shift and a global oxygen depletion \cite{Petrovskii17,sekerci15a,Sekerci18}. 
Ocean anoxia is thought to be the factor that can trigger a mass extinction and this has indeed happened several times during the deep past \cite{Wignall96,Erwin14,Song14,Sudakow22}. 
Thus, better understanding of the pathways leading to the global anoxia as well as the identification of possible early signs of the approaching catastrophe are obviously problems of literally vital importance. In turn, it requires a better understanding of the coupled plankton-oxygen dynamics in the ocean and, arguably, mathematical modelling is a powerful research approach to facilitate it. 

Mathematical models of plankton dynamics are abundant in the literature, e.g.~see \cite{Beltrami,Beltrami_Carroll,Brindley96,Steele74,Steele92,Cushing,Truscott94a,Truscott94b}. Earlier modelling studies of the plankton dynamics were mostly concerned with the temporal and spatio-temporal dynamics of the coupled phytoplankton-zooplankton system \cite{Steel78,Steele92,Medvinsky02,Malchow00}, with a particular focus on plankton patchiness and plankton blooming~\cite{Beltrami,Truscott94b,Scheffer00,Scheffer97,Martin03}. 
In \cite{Scheffer91a,Petrovskii00,Petrovskii01TPB,Petrovskii02b,Rinaldi93Eco}, conceptual two- and three-component mathematical models were considered to reveal the role of various internal and external factors and to demonstrate different routes to plankton pattern formation. 
More recently, there has been growing attention to possible links between the dissolved oxygen, plankton dynamics and the climate change \cite{Charlson87,Sarmiento98}, which facilitated further progress in mathematical modelling of marine ecosystems \cite{Hays05}. In particular, it was shown in \cite{Petrovskii17,sekerci15a,Sekerci15b} that sustainable oxygen production by marine phytoplankton can be severely disrupted by a gradual increase in the average water temperature. In turn, this may eventually lead to a global anoxia and it was argued in \cite{Petrovskii21} that the observed decrease in the oxygen stock in the ocean \cite{Schmidtko17,Breitburg18,Battaglia17} and the slow gradual decrease in the amount of atmospheric oxygen \cite{Martin17} that has occurred over the last few decades may be early signs of the approaching catastrophe. 

Remarkably, the amount of dissolved oxygen does not only change with time, it also depends on space. 
The spatial distribution of oxygen in the ocean is distinctly heterogeneous \cite{Ito17,Richardson17}, sometimes resulting in the formation of large stable areas or `patches' where the dissolved oxygen concentration is much lower than the average. Such a patch is referred to as Oxygen Minimum Zone (OMZ) or the dead zone \cite{Breitburg18,Watson17,Stramma08}. They were discovered in different parts of the world ocean, e.g.~in the subsurface waters of the Arabian Sea and in the eastern boundary upwelling regions of the tropical oceans of California, Peru, and Namibia \cite{Morrison99,Stramma08}. The existence of OMZ has a significant effect on the marine species abundance and the aquatic food chains \cite{Diaz}. There is evidence that some zooplankton species may have a capacity to adapt to oxygen-deficient environments \cite{Wishner}. However, a significant drop in the dissolved oxygen level eventually results in the formation of a dead zone, so that the majority of marine life either dies or leaves the area \cite{Diaz,Watson17}. 

Interestingly, over the last several decades the OMZs have been growing in size. In particular, a rapid expansion in OMZs in the eastern Pacific and northern Indian oceans is well documented \cite{Stramma08,Breitburg18}. Because of the OMZ's detrimental effect on the corresponding local marine ecosystem, this is becoming a grave concern for ecology and conservation as well as some industries, e.g.~fishery. 
Moreover, there is some theoretical evidence that the OMZ expansion may be an early warning signal of the approaching global anoxia \cite{Alhassan23}. 
The expansion of OMZ is thought to be caused by various factors, ultimately linking it to the global warming and to the human interference through the perturbation of the ocean's biogeochemical cycles, in particular the CO$_2$ cycle, although the issue as a whole remains controversial \cite{Keeling,Lenton08}. 

In this study, we address the phenomena of the OMZ formation and growth theoretically based on the earlier conceptual modelling approach that considers the variations (in particular, a decrease) in the dissolved oxygen level as a inherent property of the coupled plankton-oxygen dynamics in the ocean, not necessarily an effect of exogenous factors \cite{sekerci15a,Petrovskii17,Sekerci18}.
Our updated mathematical model incorporates two important factors that were overlooked in the earlier studies. One such factor is the nonlinear mortality of zooplankton; it takes into account a combined effect of the zooplankton intraspecific competition, cannibalism, and the zooplankton consumption by its predators from a higher trophic level (e.g.~fish) \cite{Steele92,Scheffer91a}. It is well known that the nonlinear mortality rate can change the system's dynamical properties significantly \cite{Truscott94b,Bazykin98,Chowdhury23a} but its potential effect on the plankton-oxygen dynamics has never been investigated. 

The second factor is the existence of different timescales in the plankton-oxygen dynamics. 
Indeed, it is a common observation that the zooplankton growth rate is usually much lower than that of phytoplankton. Correspondingly, the typical time (timescale) of changes in the phytoplankton density is considerably shorter (sometimes by an order of magnitude) than that of zooplankton. 
In this study, we therefore assume that the production of oxygen and the phytoplankton growth, on the one hand, and the zooplankton response, on the other hand, happen on fast and slow timescales, respectively. The presence of different timescales in a dynamical system can make its properties much more complicated, e.g.~to bring bifurcations, coexistence of multiple attractors, complex oscillations, long transients and pattern formation that would not be there otherwise \cite{Rinaldi92,Kuehn15,Poggiale20,Sadhu21,Chowdhury23b}. It is therefore a very relevant question as to how the existence of multiple timescales can modify the plankton-oxygen dynamics, in particular in the context of the OMZ formation and growth. 

The paper is organised as follows. In the next section, we describe our mathematical model and investigate its basic properties such as the existence and stability of the steady states. In Sections \ref{sec:locvsglob} and \ref{sec:slowfast}, we consider the properties of the nonspatial model to reveal, respectively, the effect of the nonlinear mortality and the different timescales. We identify conditions when the system can undergo a regime shift that may correspond to  catastrophic changes in the real world.
In Section \ref{sec:spatial}, we consider the properties of the spatially explicit system, with a special focus on the dynamical regimes resulting in the pattern formation, in particular those that can be interpreted as the formation and/or expansion of Oxygen Minimum Zones. Finally, in Section \ref{sec:conclus} we summarise and discuss our results.

\section{The non-spatial system}\label{sec:nonspat}

We consider a conceptual mathematical model that explicitly includes only phytoplankton, zooplankton, and dissolved oxygen. Oxygen is produced by phytoplankton in photosynthesis and consumed by both phytoplankton and zooplankton as needed for their metabolism. 
%\textcolor{blue}{Interaction between phyto- and zooplankton is considered a prey-predator type interaction. Depending on the observation of the controlled experiments, the authors in \cite{Steele81,Medvinsky02} argue that it is almost impossible to portray all the details of the plankton ecosystem in a single model. He proposed a simple plankton model where the predation of phytoplankton by zooplankton follows a Holling type II functional response. The growth of phytoplankton is further correlated with the rate of photosynthesis \cite{Franke99}, which depends on the concentration of available oxygen. Furthermore, the rate of change in the concentration of the dissolved oxygen depends on the phytoplankton mass, and the transport of oxygen from the phytoplankton cells to the surrounding water.}
In the zero-dimensional (nonspatial) case, the model is given by the following three equations~\cite{sekerci15a,Petrovskii17}:
\begin{subequations} \label{eq:temporal}
\begin{align}
\frac{dc}{dt} &= \frac{Au}{c+1} - \frac{\delta uc}{c+c_2}- \frac{\nu cv}{c+c_3} - c = F(c,u,v), \label{subeq1a} \\
\frac{du}{dt} &= \left(\frac{B c}{c+c_1} - u \right)u - \frac{u v}{u+h} - \sigma u=G(c,u,v), \label{subeq1b}  \\
\frac{dv}{dt} &= \varepsilon\Big(\frac{\eta c^2}{c^2+{c_4}^2}\frac{uv }{u+h} - \mu_1 v-\mu_2 v^2\Big)=\varepsilon H(c,u,v), \label{subeq1c}  
\end{align}
\end{subequations}
%(cf.~\cite{sekerci15a,Petrovskii17,Sekerci18}) 
where $c,\,u$ and $v$ denote, respectively, the concentration of dissolved oxygen, phytoplankton, and zooplankton densities in appropriately chosen dimensionless variables~\cite{sekerci15a,Petrovskii17} and $t$ is dimensionless time. The first term in Eq.~\eqref{subeq1a} describes the rate of oxygen production in photosynthesis (see~\cite{sekerci15a} for more details) and the second and third terms describe the oxygen consumption by phyto- and zooplankton, respectively, which is assumed to be described by the Monod type kinetics. The first term in Eq.~\eqref{subeq1b} describes the phytoplankton multiplication; based on earlier work~\cite{Steele81,Steele92,Franks02}, we consider it as the logistic growth. The second term in Eq.~\eqref{subeq1b} quantifies the phytoplankton grazing by zooplankton~\cite{Franks02}, and the third term describes the phytoplankton natural mortality. In Eq.~\eqref{subeq1c}, the first term in the brackets describes the zooplankton growth (with the food assimilation efficacy being assumed to depend on the level of dissolved oxygen~\cite{sekerci15a}), the second term describes the zooplankton natural mortality and the third (quadratic) term describes a combined effect of the competition and the predation by species from higher trophic levels (not included into the model explicitly)~\cite{Truscott94a}. Here parameter $A$ is the per capita oxygen production rate, $B$ is the per capita phytoplankton growth rate, $\sigma$ and $\mu_1$ are natural mortality rates for phyto- and zooplankton, respectively, $\mu_2$ quantifies the strength of nonlinear zooplankton removal. 
The meaning of other parameters in Eqs.~\eqref{eq:temporal} is straightforward. For more details, including biological justification of all terms in the right-hand side of Eqs.~\eqref{eq:temporal}, see~\cite{sekerci15a,Petrovskii17}. 

For the reasons mentioned in the introduction, we introduce additional parameter $0<\varepsilon<1$ that quantifies the difference in the phyto- and zooplankton characteristic timescales; in most cases below, we will assume $\varepsilon\ll 1$. Note that, compared to the original model proposed in \cite{sekerci15a,Petrovskii17}, Eqs.~\eqref{eq:temporal} include two essentially new elements, i.e.~the quadratic term and a small parameter (cf. Eq.~\eqref{subeq1c}), which makes model \eqref{eq:temporal}somewhat more realistic.

\subsection{Steady States Analysis} 

To explore the dynamics of the temporal model, we study all possible equilibrium points (steady states) of the system \eqref{eq:temporal} and their stability. The system has a total extinction state given by $E_0=(0,0,0).$ To study the dynamics of the system \eqref{eq:temporal} around $E_0$ we linearize around $E_0$ and obtain the Jacobian matrix 
\begin{equation}
  J_{E_0}= \begin{pmatrix}
  -1&A&0\\0&-\sigma&0\\0&0&-\varepsilon\mu_1
    \end{pmatrix}.
\end{equation}
All the eigenvalues of the above matrix $J_{E_0}$ are real and negative. Therefore, the total extinction state $E_0$ is always stable. In \cite{sekerci15a} the authors showed that under some parametric restrictions, the system can have two zooplankton free equilibria. Since the introduction of the quadratic term in \eqref{subeq1c} does not affect the zooplankton free equilibrium states, as the zooplankton free state is the form $(\Bar{c},\Bar{u},0)$ where $\Bar{c}$ is the root of the quartic equation 
    \begin{equation}\label{boundarystate-c}
        \begin{aligned}
         \Bar{c}^4-(\delta(\sigma-B)-(c_1+c_2+1))\Bar{c}^3 &- (A(B-\sigma)+(\delta\sigma-c_2-1)c_1-B\delta+\delta\sigma-c_2)\Bar{c}^2 -\\ &(((B-\sigma)c_2-\sigma c_1)A+\delta\sigma c_1-c_1c_2)\Bar{c} + A\sigma c_1c_2=0,
        \end{aligned}
\end{equation}
 and 
 \begin{equation}\label{boundarystate-u}
 \Bar{u} = \frac{\Bar{c}(B-\sigma)-c_1\sigma}{\Bar{c}+c_1}.
 \end{equation}
 Since the above equation is a fourth-order polynomial, analytical determination of the equilibrium points is nearly impossible. We choose suitable numerical values of the parameters to obtain feasible zooplankton free equilibrium states. Throughout the paper, we fix the parameter values at
\begin{equation}\label{parameters}
 A = 4,\,B=3,\,\sigma = 0.1,\, c_1 = 0.7,\, c_2 = 1,\, c_3 = 1,\, c_4 = 1,\, \eta = 0.7,\, \delta = 1,\, \gamma = 0.01,\, h = 0.1.
\end{equation}
and suitably vary $\mu_1,\,\mu_2,$ and $\varepsilon.$ For all the values of $\mu_1$ and $\mu_2$ and parameters fixed at \eqref{parameters} the two feasible zooplankton free equilibrium points are given by $E_1=(0.0258,0.0067,0)$ and $E_2=(1.712,2.029,0).$ Among the two zooplankton free equilibrium points, one is always a saddle point while the other can be either stable or unstable depending on the parameter values. For our choice of parameter values \eqref{parameters} both $E_1$ and $E_2$ are unstable (saddle) with a 2-dimension stable manifold and a 1-dimensional unstable manifold. We plot the $c$-component of $E_1$ and $E_2$ in Fig.~\ref{fig:coexistence_eq} with red broken lines. Note the lines are horizontal, because the corresponding steady state values do not depend on $\mu_1$ or $\mu_2$, as is obvious from Eqs.~(\ref{boundarystate-c}-\ref{boundarystate-u}). The system does not possess any other feasible boundary equilibria in this parametric regime. Note that a nontrivial oxygen free equilibrium is not possible in our model (which agrees with biological reasons). Indeed, setting $c\equiv 0$ in Eq.~(1a) immediately leads to $u\equiv 0$, which in turn leads to $z\equiv 0$.

The coexistence equilibrium of the system \eqref{eq:temporal} is denoted by $E_*=(c_*,u_*,v_*)$ where $$v_*=\frac{1}{\mu_2}\Big(\frac{\eta c_*^2}{c_*^2+c_4^2}\frac{u_*}{u_*+h}-\mu_1\Big)$$ and $c_*,\,u_*$ can be obtained by simultaneously solving the quartic and quadratic equations respectively

\begin{eqnarray*}
 c_*^4 + (\delta u_* + \nu v_*+ c_2 + c_3+1)c_*^3 - ((A-\delta(1+c_3))u_* - (1+c_2)\nu v_* - (c_2(1+c_3)+c_3))c_*^2 \\
   - (Au_*(c_2+c_3)-(\delta u_*c_3+\nu v_*c_2+c_2c_3))c - Au_*c_2c_3 = 0,\\
   u_*^2(c_*+c_1) - ((B-h-\sigma)c_*-(h+\sigma)c_1)u_* - (Bh-h\sigma-v_*)c_* + (h\sigma+v_*)c_1 = 0.
\end{eqnarray*}
\noindent 
Because of the complexity of the simultaneous algebraic equations we obtain the coexistence equilibrium points using numerical simulations. The parameter values are fixed at \eqref{parameters} and we choose $\mu_1,\ \mu_2$ as the bifurcation parameters. Note that for $\mu_1>0,\,\mu_2=0,$ the model was rigorously studied in \cite{sekerci15a}. However, we deal with other cases in this work. For $\mu_1=0,\,\mu_2>0,$ we found that the system \eqref{eq:temporal} has two coexistence equilibrium points which disappears via saddle-node bifurcation for sufficiently smaller values of $\mu_2$. If we denote the saddle-node bifurcation threshold by $\mu_2^{\mathrm{SL}},$ then for $\mu_2>\mu_2^{\mathrm{SL}}$ the system has two coexistence equilibrium points and for  $\mu_2<\mu_2^{\mathrm{SL}}$ there are no feasible coexistence equilibrium. For $\mu_1,\,\mu_2>0,$ the system has a unique coexistence equilibrium throughout the parametric regime of $\mu_2,$ whenever they are feasible. 
The number of coexistence equilibrium points and their stability is determined by the combination of the parameters $\mu_1$ and $\mu_2.$ The Jacobian matrix evaluated at the coexistence equilibrium $E_*$ is 
\begin{equation}\label{eq:Jacobi}
    J_{E_*}=\begin{pmatrix}
    J_{11}&J_{12}&J_{13}\\ J_{21}&J_{22}&J_{23}\\J_{31}&J_{32}&J_{33}\\
    \end{pmatrix}
\end{equation}
and the corresponding characteristic equation is given by
$$\lambda^3+p_2\lambda^2+p_1\lambda+p_0=0.$$
The coefficient are $$p_2=-\mathrm{tr}(J_{E_*}),\,\,p_1=J_{11}^{[1]}+J_{22}^{[2]}+J_{33}^{[3]},\,\,p_0=-\det(J_{E_*})$$ where $J_{ij}$ is the $(i,j)$-th entry of  $J_{E_*},$ and $J_{ii}^{[i]}$ is the cofactor of $J_{ii}.$  By Routh-Hurwitz conditions the coexistence equilibrium point $E_*$ is stable if the following conditions hold $$p_0>0,\,p_2>0\,\,\text{and}\,\,p_1p_2>p_0.$$
To observe the change in system's dynamics with the introduction of intraspecific competition among the zooplankton, we choose $\mu_2$ to be the bifurcation parameter. Thus keeping all the parameters fixed, $p_1,\,p_2,\,p_0$ are functions of $\mu_2$ only. The coexistence equilibrium $E_*$ loses its stability and exhibits oscillatory dynamics through Hopf bifurcation at $\mu_2^{\mathrm{H}},$ which is obtained by solving  $$p_1(\mu_2^{\mathrm{H}})p_2(\mu_2^{\mathrm{H}}) = p_0(\mu_2^{\mathrm{H}}).$$ 
This expression is not mathematically tractable, thus we numerically obtain the Hopf bifurcation threshold $\mu_2^{\mathrm{H}}= 0.35405$ (upto five decimal place) for the parameter values \eqref{parameters}, and $\mu_1=0.05,\,\varepsilon=1.$ The Hopf bifurcation can be either supercritical or subcritical depending on the magnitude of the parameters $\mu_1$ and $\mu_2.$ For a fixed $\mu_1,$ with an increase in the intraspecific competition (as quantified by $\mu_2$), \textcolor{black}{the equilibrium state of the system} changes its stability from unstable to stable. 
\begin{figure}[ht!]
    \centering
   \includegraphics[scale=0.5]{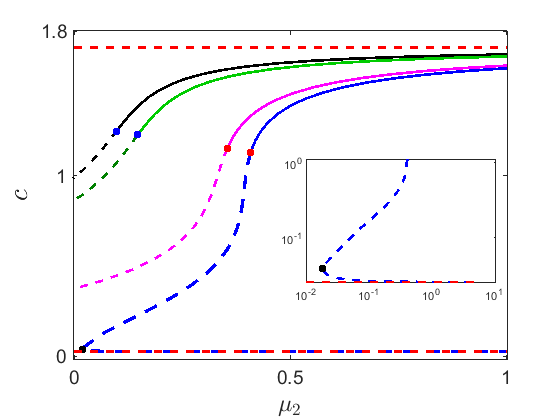}
 \caption{\textcolor{black}{The plot of $c$-component of the coexistence equilibrium for the parameter values given in \eqref{parameters} as a function of parameter $\mu_2$ for a few different values of $\mu_1$:  $\mu_1=0$ (blue), $\mu_1=0.05$ (magenta), $\mu_1=0.25$ (green), and $\mu_1=0.3$ (black). The black dot represents the saddle-node bifurcation threshold $(c^{\mathrm{SL}}, \mu_2^{\mathrm{SL}}).$ The inset shows the nature of the equilibrium branch near the saddle-node bifurcation threshold in log-log scale. The red horizontal lines represent the $c$-component of the two zooplankton-free equilibrium points.  The red dots on the equilibrium branches represent subcritical Hopf bifurcation threshold, while the blue dots on the equilibrium branches represents supercritical Hopf bifurcation threshold. } }    \label{fig:coexistence_eq}
\end{figure}    
The broken line in Fig.~\ref{fig:coexistence_eq} represents the unstable branch of the equilibrium and the continuous line represents the stable branch. \textcolor{black}{The red and blue dots on the equilibrium branch are the subcritical and supercritical Hopf bifurcation thresholds, respectively.} For a relatively higher value of $\mu_2,$ both the extinction state $E_0$ and the coexistence state $E_*$ is stable, and the system exhibits bi-stability. For $\mu_1=0$ the lower branch of the coexistence equilibrium is saddle and it acts as a separatrix between the basin of attraction of $E_*$ and $E_0.$ However, for $\mu_1>0,$ their basins are separated by the unstable manifold of the saddle boundary equilibrium point. For a fixed $\mu_2$ we observe an increase in the oxygen concentration with an increasing mortality rate of the zooplankton (see Fig.\ref{fig:coexistence_eq}). Therefore, an increase in the zooplankton linear mortality rate (as quantified by $\mu_1$) and an increase in the nonlinear mortality rate (due to the intraspecific competition and/or predation by higher trophic levels \cite{SteeleHenderson92}, as quantified by $\mu_2$) lead to an increase in the stable oxygen level.

\section{Local and global dynamics}\label{sec:locvsglob}

The dynamics of the system \eqref{eq:temporal} is determined by its local and global bifurcation structure. However, any comprehensive analytical analysis of the bifurcations is hardly possible for this model due to its algebraic complexity. Instead, we choose $\mu_1$ and $\mu_2$ as bifurcation parameters and obtain two corresponding one parametric bifurcation diagrams; see Fig.~\ref{fig:local_dynamics_temporal}. These diagrams readily reveal the various Hopf bifurcation scenarios of the coexistence equilibrium. We further show the possibility of the heteroclinic bifurcation and saddle-node bifurcation of limit cycles.
\begin{figure}[ht!]
    \centering
    \subfigure[$\mu_2=0$]{\includegraphics[scale=0.35]{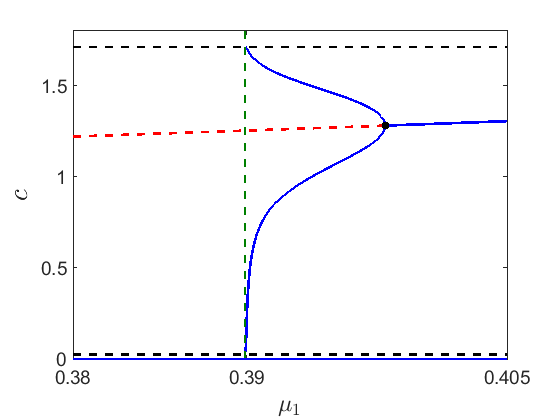}}
     \subfigure[$\mu_1=0$]{\includegraphics[scale=0.35]{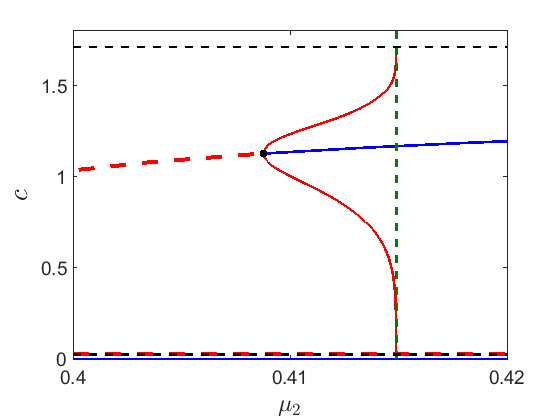}}
     \subfigure[$\mu_1=0.24$]{\includegraphics[scale=0.35]{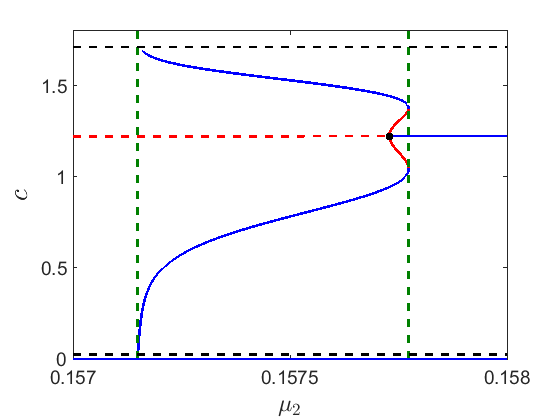}}
    \caption{Bifurcation diagram of the system \eqref{eq:temporal} at the parameter values \eqref{parameters} for different combination of $\mu_1$ and $\mu_2.$ The stable equilibrium and limit cycles are marked in blue. The unstable cycle and coexistence equilibrium are marked in red (continuous and broken respectively). The $c$-component of the boundary equilibrium points is shown by broken black lines. The global bifurcation thresholds are marked in vertical broken green line. The black dot represents the Hopf bifurcation threshold (subcritical and supercritical).}    \label{fig:local_dynamics_temporal}
\end{figure}

In Figure \ref{fig:local_dynamics_temporal}, 
the Hopf bifurcation thresholds (supercritical and subcritical) are marked by black dots and the thresholds for the global bifurcations (heteroclinic and saddle-node bifurcation of limit cycles) are marked by green vertical broken lines. The stable and unstable cycles (and equilibrium) are shown in blue and red colour, respectively. We choose $\mu_1$ or $\mu_2$ as bifurcation parameters to observe the change in the system's dynamics. Note that, the zooplankton free equilibrium points $E_1$ and $E_2$ are independent of the bifurcation parameters. Thus, the stability of these equilibrium points does not depend on $\mu_1$ and $\mu_2.$ For our investigation, we choose other parameters as in \eqref{parameters}, so that that both $E_1$ and $E_2$ are unstable. we plot the $\Bar{c}$-components of the unstable zooplankton free equilibrium points of the form $(\Bar{c},\Bar{u},0)$ in a black broken line. 

In the absence of intra-specific competition among the zooplankton and neglecting the effect of higher trophical levels (that is, for $\mu_2=0$), the system has a unique coexistence equilibrium, which loses its stability through Hopf bifurcation at $\mu_1^{\mathrm{H}}=0.398$ (considering all the other parameters are fixed at \eqref{parameters}). The Hopf bifurcation is supercritical as the first Lyapunov coefficient is $l_1= -0.358 (<0)$. The unique coexistence equilibrium point is stable for $\mu_1>\mu_1^{\mathrm{H}}$ and unstable for $\mu_1<\mu_1^{\mathrm{H}}.$ A small amplitude stable limit cycle originates from $\mu_1^{\mathrm{H}}.$ The amplitude of the stable cycle increases with decreasing $\mu_1$ till it hits the boundary equilibrium points and disappears via a heteroclinic bifurcation (marked in green). Beyond this threshold, the coexistence equilibrium is unstable and the extinction state $E_0$ is the only global attractor. On decreasing the mortality rate from $\mu_1^\mathrm{H}$ the concentration of the zooplankton increases. This increases the oxygen consumption by zooplankton and also decreases phytoplankton production such that beyond a threshold the system collapse.

On the contrary, if we consider the case where the zooplankton mortality is primarily due to the combined effect of the competition and 
predation by species from higher trophic levels, then one could neglect the natural mortality~\cite{SteeleHenderson92}, so that $\mu_1=0$. Using other parameter values as in~\eqref{parameters}, the Hopf bifurcation occurs at $\mu_2^{\mathrm{H}} = 0.408.$ Since the first Lyapunov coefficient is $l_1=1.065(>0),$ the Hopf bifurcation is subcritical. Among the two coexistence equilibria as obtained in Fig.~\ref{fig:coexistence_eq}, one is always saddle, the other is unstable for $\mu_2<\mu_2^{\mathrm{H}}$ and stable for $\mu_2>\mu_2^{\mathrm{H}}.$ An unstable cycle emerges from the subcritical Hopf bifurcation, increases in size in a narrow domain for $\mu_2>\mu_2^{\mathrm{H}},$ till it disappears via global bifurcation (see Fig.~\eqref{fig:local_dynamics_temporal}(b)). Beyond this, the stable extinction state $E_0$ coexists with a stable coexistence equilibrium. We now choose $\mu_1=0.24.$ Then, the subcritical Hopf bifurcation occurs at $\mu_2^{\mathrm{H}}=0.1577.$ Here the unstable cycle that emerged through Hopf bifurcation is surrounded by a stable limit cycle. The stable and unstable cycle exists in a narrow range of $\mu_2$ and collides at saddle-node bifurcation of limit cycle ($\mu_2=0.1578).$ On the other hand, for $\mu_2<\mu_2^{\mathrm{H}},$ the unstable equilibrium is surrounded by a stable limit cycle. The amplitude of the stable cycle increases and disappears through heteroclinic bifurcation at $\mu_2=0.15715$ (see Fig.~\eqref{fig:local_dynamics_temporal}(c)).

\section{Slow and fast dynamics}\label{sec:slowfast}

The system \eqref{eq:temporal} for $0<\varepsilon\ll 1$ is referred to as the singularly perturbed system. The timescale parameter $\varepsilon$ signifies a clear distinction between two timescales: slow and fast. The change in the zooplankton density occurs at a much slower rate as compared to the change in oxygen and phytoplankton density. We denote the time $`t$' in system \eqref{eq:temporal} as the fast time and the system \eqref{eq:temporal} is with respect to the fast timescale.\\

Introducing the slow timescale $\tau=\varepsilon t,$ we obtain the following slow subsystem:

\begin{equation} \label{eq:temporal_slow}
    \begin{aligned}
        \varepsilon   \frac{dc}{d\tau} &= \frac{Au}{c+1} - \frac{\delta uc}{c+c_2}- \frac{\nu cv}{c+c_3} - c = F(c,u,v), \\
 \varepsilon   \frac{du}{d\tau} &= \left(\frac{B c}{c+c_1} - u \right)u - \frac{u v}{u+h} -    \sigma u=G(c,u,v), \\
     \frac{dv}{d\tau} &= \Big(\frac{\eta c^2}{c^2+{c_4}^2}\frac{uv }{u+h} - \mu_1 v-\mu_2 v^2\Big)= H(c,u,v),
    \end{aligned}
\end{equation}
We analyze the above slow and fast systems with the help of geometric singular perturbation theory \cite{Kuehn15,Fenichel}. The basic idea behind this approach was to decompose the slow and fast systems in its limiting systems (i.e for $\varepsilon=0$) and study the dynamics of the respective subsystems. In the singular limit $\varepsilon\rightarrow0,$ we obtain the fast subsystem (oxygen-phytoplankton) of system \eqref{eq:temporal} as follows

\begin{equation}
    \begin{aligned}
     \frac{dc}{dt} &= \frac{Au}{c+1} - \frac{\delta uc}{c+c_2}- \frac{\nu cv_0}{c+c_3} - c, \\
  \frac{du}{dt} &= \left(\frac{B c}{c+c_1} - u \right)u - \frac{u v_0}{u+h} - \sigma u, \\   
    \end{aligned}
\end{equation}
with $v=v_0$ (constant zooplankton density). Also, letting $\varepsilon=0,$ in \eqref{eq:temporal_slow} we obtain the slow subsystem or the reduced system  as
\begin{equation}\label{eq:slow_sub}
         F(c,u,v) =0,\,\,
         G(c,u,v) =0,\,\,
         \frac{dv}{d\tau} = H(c,u,v),
\end{equation}
where $F,G,H$ are given above. The set $$C_0 = \{(c,u,v)\in \mathbb{R}^3: F(c,u,v)=0=G(c,u,v) \}.$$
is called critical manifold, which  is the collection of all equilibrium points or curves of the fast subsystem. It can be divided into two parts: trivial critical manifold $C_0^0$ and non-trivial critical manifold $C_0^1$ such that $C_0=C_0^0\cup C_0^1.$ The trivial critical manifold is given by the line $$C_0^0 = \{(c,u,v)\in \mathbb{R}^3: c=0,\,u=0,\,v\in \mathbb{R}\}$$ and the non-trivial critical manifold $C_0^1$ is the curve of intersection of two surfaces given by $$v=(u+h)\Big(\frac{Bc}{c+c_1}-u-\sigma\Big),\,\,v=\frac{(c+c_3)}{\nu}\Big(\frac{Au}{c+1} - \frac{\delta uc}{c+c_2}-c\Big).$$
The slow subsystem \eqref{eq:slow_sub} represents the slow change in the zooplankton density over the critical manifold. Considering the slow variable $v$ as the bifurcation parameter, we numerically obtain the critical manifold (black) as shown in Fig.~\eqref{fig:trajectory_slow_fast}. Therefore, along the critical manifold, the zooplankton density changes slowly. Since the two surfaces intersect along a curve, we find the extrema (fold point) of the curve (if any). A point $P\in C_0$ is a fold point of the critical manifold if the fast subsystem exhibits a fold bifurcation. In other words, if we consider the Jacobian matrix of the fast subsystem
\begin{equation}
    \mathcal{J}=\begin{pmatrix}F_c&F_u\\G_c&G_u
    \end{pmatrix},
\end{equation}
then at the fold point $P$, $\mathrm{rank}(\mathcal{J})=1$ and the critical manifold is non-hyperbolic. It divides the critical manifold into two halves namely attracting $C_0^a$ and repelling $C_0^r$ sub-manifold. A point $p\in C^a_0$ if both the eigenvalues of the matrix $\mathcal{J}$ evaluated at $p$ has negative real parts, whereas a point $q\in C_0^r$ if atleast one of the eigenvalues of the matrix $J$ evaluated at $p$ has positive real part. To determine the stability of the trivial manifold, we evaluate the matrix $\mathcal{J}$ along $C_0^0$ and thus obtain
\begin{equation}
    \mathcal{J}_{C^0_0}=\begin{pmatrix}
    -1-\frac{\nu v}{c_3}&A\\0&-\frac{v}{h}-\sigma
    \end{pmatrix}.
\end{equation}
Since all the parameters involved in the system are positive, the eigenvalues are given by $$-1-\frac{\nu v}{c_3}<0,\,\,\,-\frac{v}{h}-\sigma<0,$$ for $v\ge0.$ Thus the trivial manifold $C_0^0$ is stable for $v\ge0.$ Fenichel's theorem \cite{Fenichel} state that the normally hyperbolic attracting and repelling sub-manifolds, $C_0^a$ and $C^r_0$ respectively, obtained for $\varepsilon=0,$ perturb to locally invariant attracting and repelling sub-manifolds $C^a_{\varepsilon}$ and $C^r_{\varepsilon}$ respectively, for $\varepsilon>0.$ Therefore, the dynamics of the full system \eqref{eq:temporal} or \eqref{eq:temporal_slow} can be approximated by studying the dynamics of the subsystems obtained for $\varepsilon=0.$

The slow flow on the critical manifold $C_0^1$ is given by the slow subsystem \eqref{eq:slow_sub}. We differentiate $F(c,u,v)=0$ and $G(c,u,v)=0$ implicitly with respect to $`\tau$' along the critical manifold, and obtain the dynamics on the critical manifold. This is governed by the following system of equations
\begin{equation}
    \begin{aligned}
         \frac{dc}{d\tau}  = -\frac{F_vG_u-F_uG_v}{F_cG_u-F_uG_c}H,\,\,
         \frac{du}{d\tau}  = -\frac{F_vG_c-F_cG_v}{G_cF_u-G_uF_c}H ,\,\,
         \frac{dv}{d\tau} = ~~H,
    \end{aligned}
\end{equation}
with suitable initial condition $(c_0,u_0,v_0)\in C_0^1.$ The slow flow has a singularity whenever $G_cF_u-G_uF_c=0,$ which holds at the fold point $P.$ Thus, the solution blows up at this point. Whenever $F_vG_u-F_uG_v\ne0$ or $F_vG_c-F_cG_v\ne0,$ the fold point is called the jump point, and the trajectory jumps from the proximity of the fold point to another attracting critical manifold. However, when both $F_vG_u-F_uG_v=0$ and $F_vG_c-F_cG_v=0,$ the fold point is called canard point. At this point, the trajectory can pass through the proximity of the fold point and follow the repelling manifold for $O(1)$ time.

\noindent The coexisting equilibrium loses its stability and exhibits oscillatory dynamics through Hopf bifurcation (supercritical or subcritical), which we discussed in the previous section. In a classical slow-fast setting, the small cycles originating from Hopf bifurcation bifurcates to canard cycles (with or without head) and further to relaxation oscillation, thus exhibiting canard explosion \cite{Krupa01b}. These cycles are composed of slow and fast segments, where the slow flow occurs along both attracting and repelling sub-manifold of the critical manifold. It exhibits fast flow when the trajectory jumps to another stable portion of the critical manifold. However, for the system \eqref{eq:temporal}, small canard cycles (without head) emerge from Hopf bifurcation. The amplitude of the cycles increases in a narrow parametric range, eventually leading to complete extinction. We prove this fact in the following theorem.

\begin{theorem}
Assume the fold point $P$ is a canard point for $\varepsilon>0.$ Then the system \eqref{eq:temporal} has small amplitude canard cycles (without head) originating from Hopf bifurcation but there does not exist any relaxation oscillation.
\end{theorem} 

\begin{proof}
In the slow-fast setting, we denote the Hopf bifurcation as singular Hopf bifurcation since the eigenvalues of the Jacobian matrix of the system \eqref{eq:temporal} evaluated at $E_*$ has purely imaginary complex eigenvalue of the form 
$$\lambda_{1,2} = \pm ~i \omega(\mu_2^{\mathrm{H}},\varepsilon)$$ such that $\lim_{\varepsilon\rightarrow0} \omega(\mu_2^{\mathrm{H}},\varepsilon)=0.$ The singular Hopf bifurcation occurs at $O(\varepsilon)$ from the fold point P. We assume $F_vG_u-F_uG_v=0$ and $F_vG_c-F_cG_v=0$ such that the fold point is the canard point. The small limit cycle originating from Hopf bifurcation grows in size through a sequence of canard cycles. With the decrease in $\varepsilon,$ the amplitude of the cycle increases, and after a certain threshold, the trajectory jumps from the vicinity of the fold point $P$ close to $C_0^0.$
The equilibrium point $E_0$ lies on the trivial critical manifold. The eigenvalues of the Jacobian matrix $J_{E_0}$
are $-1,\,-\sigma,\,\,-\varepsilon\mu_1$ and the corresponding eigenvectors are $$\left(1,0,0\right),\,\left(1,\frac{1-\sigma}{A},0\right),\,\ \text{and}\,\ \left(0,0,1\right).$$ Therefore the critical manifold $C_0^0$ coincides with the eigenvector. Thus, any trajectory on $C_0^0$ converges to $E_0.$ We cannot construct any singular orbit consisting of concatenated slow segments on $C_0^1$ and $C_0^0,$ and fast fibers while leaving the respective manifolds. Hence, the global return mechanism, which is necessary for the existence of classical relaxation oscillation, fails as all the trajectories converge to the stable equilibrium $E_0$. 
\end{proof}

\begin{figure}[ht!]
    \centering
    \subfigure[]{\includegraphics[scale=0.38]{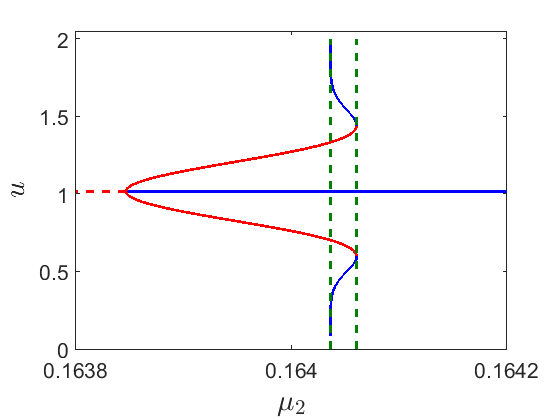}}
    \subfigure[]{\includegraphics[scale=0.38]{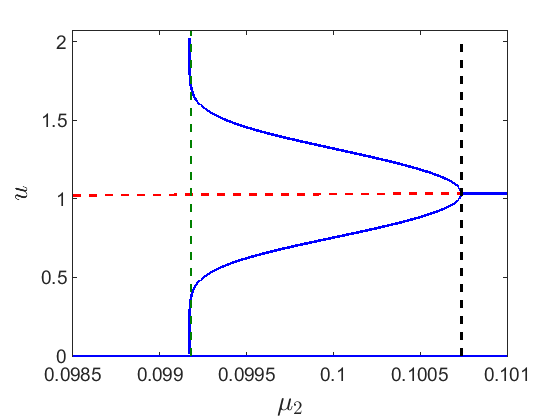}}
    \caption{\textcolor{black}{The change in the amplitude of the canard cycle emerging from the singular Hopf bifurcation with varying $\mu_2$ is shown for (a) $\mu_1=0.24,\,\varepsilon=0.5,$ and (b) $\mu_1=0.3,\,\varepsilon=0.5.$  The blue lines show the steady state value of of the coexistence equilibrium when it is stable and the maximum and minimum amplitude of the stable canard cycle when the equilibrium is unstable. The horizontal red (broken) line shows the steady state value of the coexistence equilibrium when it is unstable. The vertical black (broken) line marks the singular Hopf bifurcation threshold (occurring at $\mu_2=0.1007$) and the vertical green (broken) line at $\mu_2=0.099171$ indicate the threshold for the system collapse (plankton extinction and oxygen depletion). Other parameters of the system are given in \eqref{parameters}.}}
    \label{fig:bifurcation_canard_explosion}
\end{figure}

We illustrate this phenomenon with the help of a numerical example in Fig.~\ref{fig:bifurcation_canard_explosion}. We use the parameter values as in~\eqref{parameters} along with $\varepsilon=0.5$. Consider a hypothetical value $\mu_1=0.24.$  The effect of $\varepsilon$ on the dynamics of the system can be observed if we compare the Fig.~\ref{fig:bifurcation_canard_explosion}(a) with Fig.~\ref{fig:local_dynamics_temporal}(c). The singular Hopf bifurcation occurs at $\mu_2^{\mathrm{H}}=0.1638$ is subcritical ($l_1=0.037$). Small unstable canard cycles emerge from the canard point $P~(1.24,1.01,0.89)$, which grows in size with a slight increase in $\mu_2$ (Fig.~\ref{fig:bifurcation_canard_explosion} (a)). A large amplitude stable cycle emerges from a heteroclinic bifurcation that coexists with the stable equilibrium, separated by an unstable canard cycle. We observe that the size of the stable cycle shrinks in an extremely narrow parameter interval and disappears at a saddle node bifurcation of limit cycles. 

We now consider $\mu_1=0.3$ and $\varepsilon=0.5,$ then the singular Hopf bifurcation occurs at $\mu_2^{\mathrm{H}} = 0.1007.$ The first Lyapunov coefficient is $l_1=-1.1818$, hence the singular Hopf is supercritical. From the canard point $P~(1.255,1.035,0.898),$ small stable canard cycles originate. We show the change in the amplitude of the canard cycles with decreasing values of $\mu_2$ in Fig.~\ref{fig:bifurcation_canard_explosion}(b). This depicts that the system becomes unstable with a decrease in the strength of the intra-specific competition among the zooplankton. The transition from the stable, steady state to the oxygen-free state, indicating complete population collapse, takes place in an extremely narrow interval of the rate of intraspecific competition. That is, for $\mu_2 \in (0.099171,0.1007).$ At $\mu_2=0.09917,$ when the size of the limit cycle explodes, the trajectory converges to the origin, which is illustrated in Fig.~\ref{fig:trajectory_slow_fast}(a). The time series of the trajectory is shown in Fig.~\ref{fig:trajectory_slow_fast}(b). This implies that the system cannot further sustain any large amplitude oscillations. The rise in the amplitude of the phytoplankton level beyond a threshold can act as an indicator of population collapse. The system can therefore be driven to total extinction by pushing it far enough to reach the fold point. Along with $\mu_2,$ the timescale separation plays a critical role in identifying this narrow parametric regime. For a larger timescale separation (i.e.~smaller $\varepsilon$), the coexistence of oxygen-plankton occurs in a significantly narrower interval, and the system is more vulnerable to perturbation. 

\begin{figure}[!t]
    \centering
    \subfigure[]{\includegraphics[scale=0.45]{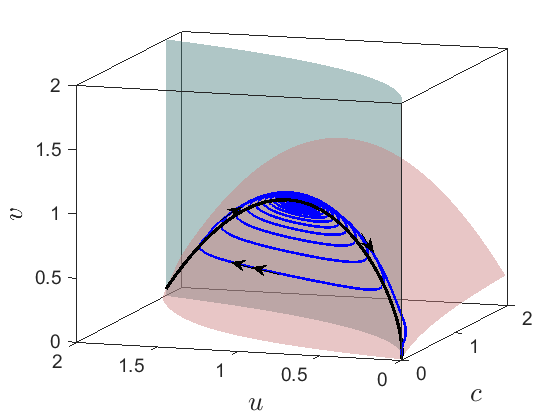}}
    \subfigure[]{\includegraphics[scale=0.43]{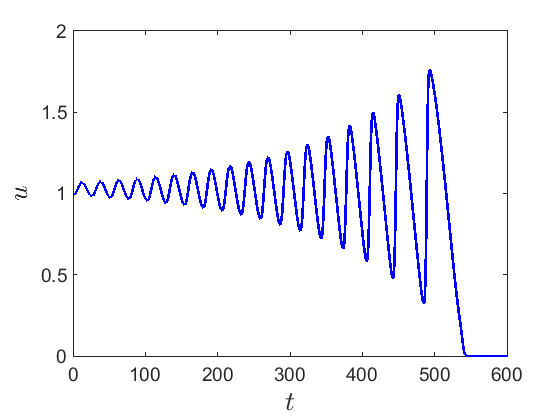}}
    \caption{(a) A trajectory (blue) converging to the origin (extinction) after oscillations of increasing amplitude obtained for $\mu_1=0.3,\,\mu_2 = 0.09917$ and $\varepsilon=0.5.$ The other system parameters are mentioned in equation \eqref{parameters}. \textcolor{black}{The surfaces $F=0$ and $G=0$ are shaded in green and brown, respectively. The black curve on the intersection of these surfaces is the critical manifold $C^1_0.$ The single and double arrows represent slow and fast motion, respectively.} (b) The corresponding dependence of the phytoplankton density on time. }
    \label{fig:trajectory_slow_fast}
\end{figure}

\section{The spatial system}\label{sec:spatial}

In the real world ocean environment, the spatial distribution of both plankton and dissolved oxygen is remarkably heterogeneous, sometimes showing the variability by an order of magnitude of even more \cite{Steele74,Levin90,Martin03,Ito17,Richardson17}. 
Correspondingly, in this section, we consider a spatially extended model \eqref{eq:temporal} where the oxygen concentration and the phytoplankton and zooplankton densities vary with both time and space. The ocean system is three-dimensional; however, in this paper, for the sake of simplicity, we only consider one horizontal spatial dimension. We regard it as the position along the ocean surface. In terms of ocean observations, it corresponds to a transect across the study area. 

Trying to keep the model as simple as possible, we avoid using explicit dependence on the vertical dimension (i.e.~the depth). Correspondingly, we use the well-mixed layer approximation~\cite{Steele81,Steele92,Franks02} to assume that the vertical distribution of plankton and oxygen is approximately uniform within the photic (upper) ocean layer where most of photosynthetic oxygen production takes place. 

The transport of any substance in the ocean takes place primarily due to the water movement. The movement in the horizontal direction occurs either due to an ocean current or marine turbulence (or their combination). Here we chose to focus on the effect of unbiased (isotropic) movement, hence taking into account only the effect of turbulence, which we describe as the turbulent diffusion quantified by a certain diffusion coefficient \cite{Monin71,Okubo80}. 

We, therefore, arrive at the following equations:
\begin{equation}\label{eq:spatio-temp_local}
\begin{aligned}
      \frac{\partial c}{dt} &= ~~D_c\frac{\partial^2{c}}{\partial{x}^2} +\frac{Au}{c+1} - \frac{\delta uc}{c+c_2}- \frac{\nu cv}{c+c_3} - c, \\
\frac{\partial u}{dt} &= ~~D_u\frac{\partial^2{u}}{\partial{x}^2} +\left(\frac{B c}{c+c_1} - u \right)u - \frac{u v}{u+h} - \sigma u, \\
 \frac{\partial v}{dt} &= ~~D_v\frac{\partial^2{v}}{\partial{x}^2} +\varepsilon\Big(\frac{\eta c^2}{c^2+{c_4}^2}\frac{uv }{u+h} - \mu_1v-\mu_2 v^2\Big).
\end{aligned}
\end{equation}
Here $c(x,t),\,u(x,t),$ and $v(x,t)$ are, respectively, the oxygen concentration and the phytoplankton and zooplankton densities at the (horizontal) location $x$ and time $t$; 
$D_c,\,D_u$ and $D_v$ denote the diffusion coefficients for oxygen, phytoplankton, and zooplankton. 
Note that phytoplankton and the dissolved oxygen can be regarded as a `passive substance', i.e.~their spatial transport is entirely determined by the water flows; hence $D_c=D_u=D_T$ where $D_T$ is the turbulent diffusion coefficient. 
However, zooplankton has a certain ability to self-motion. Combined with the effect of turbulent mixing, it can result in a value of $D_v\ne D_T$. Whether $D_v$ is larger or smaller, then depends on the zooplankton movement pattern. In case zooplankton movement is entirely random (e.g.~can be regarded as Brownian motion), then one can expect that $D_v>D_T$. In case zooplankton exhibits a homing behavior, then it is likely that $D_v<D_T$. 

Along with the temporal parameters, we now non-dimensionalize the space as $\Tilde{x}=\frac{x}{\sqrt{D_c}}$. Removing the tilde for the simplicity of the notation, we obtain the following dimensionless spatial model:
\begin{equation}\label{eq:spatio-temp_local}
\begin{aligned}
      \frac{\partial c}{dt} &= ~~\frac{\partial^2{c}}{\partial{x}^2} +\frac{Au}{c+1} - \frac{\delta uc}{c+c_2}- \frac{\nu cv}{c+c_3} - c, \\
\frac{\partial u}{dt} &= ~~\frac{\partial^2{u}}{\partial{x}^2} +\left(\frac{B c}{c+c_1} - u \right)u - \frac{u v}{u+h} - \sigma u, \\
 \frac{\partial v}{dt} &= ~~D\frac{\partial^2{v}}{\partial{x}^2} +\varepsilon\Big(\frac{\eta c^2}{c^2+{c_4}^2}\frac{uv }{u+h} - \mu_1v-\mu_2 v^2\Big),
\end{aligned}
\end{equation}
where $D = \frac{D_v}{D_c}$. Equations (\ref{eq:spatio-temp_local}) are considered inside the spatial domain $\Omega=\{x\in (0,L)\}$ where $L$ is thus the length of the domain. 

Equations (\ref{eq:spatio-temp_local}) must be complemented with the initial conditions, which we consider in the following form:
\begin{eqnarray}\label{IC}
c(x,0)\,=\,\left\{
\begin{array}{ll}
c_*+0.5, & |x-L/2|\,<\,10\\
c_*,  & \text{otherwise} \\
\end{array}\right.,\,\,\ u(x,0)\,=\,\left\{
\begin{array}{ll}
u_*+0.2, & |x-L/2|\,<\,10\\
u_*,  & \text{otherwise} \\
\end{array}\right.,
\end{eqnarray} $$v(x,0)\,=
v_*,\quad \forall x \in \Omega.$$
That is, at $t=0$ the steady state densities are perturbed within a small area at the center of the domain $\Omega$. 

For the boundary conditions, we consider the zero-flux conditions:
\begin{equation}\label{BC}
	    c_x(0,t)\,=\,c_x(L,t)= u_x(0,t)\,=\,u_x(L,t)\,=\,v_x(0,t)\,=\,v_x(L,t)\,=\,0,\,\,t\,>\,0.
\end{equation}

The model is solved numerically with $L=500.$ We use the Euler method for the temporal part and five points central difference scheme for the diffusion part along with $\Delta x=1$ and $\Delta t=0.01.$

\subsection{Turing instability}\label{Sec:Turing_instability}

To study the spatial distribution of oxygen and plankton, we start our analysis in a neighborhood of a homogeneous steady-state solution of \eqref{eq:spatio-temp_local}. Time-independent or a steady state solution $(c(x),u(x),v(x))$ of the system \eqref{eq:spatio-temp_local}-\eqref{BC} satisfies the following system of equations

\medskip

\begin{equation}
    \begin{aligned}
         &\frac{\partial^2{c}}{\partial{x}^2} +\frac{Au}{c+1} - \frac{\delta uc}{c+c_2}- \frac{\nu cv}{c+c_3} - c =0, \\
   & \frac{\partial^2{u}}{\partial{x}^2} +\left(\frac{B c}{c+c_1} - u \right)u - \frac{u v}{u+h} - \sigma u =0, \\
   &D\frac{\partial^2{v}}{\partial{x}^2} +\varepsilon\Big(\frac{\eta c^2}{c^2+{c_4}^2}\frac{uv }{u+h} - \mu_1v-\mu_2 v^2\Big) =0.
    \end{aligned}
\end{equation}

\medskip
The coexistence equilibrium $E_*$ of system \eqref{eq:temporal} corresponds to a homogeneous steady state solution of the above system. \textcolor{black}{To study the dynamics of the above system near the homogeneous steady-state solution $E_*,$ we give small heterogenous perturbation as 
\begin{equation}\label{eq:hetero_perturb}
    c(x,t)=c_*+\xi_1 e^{\lambda t}\cos kx,\,u(x,t)=u_*+\xi_2 e^{\lambda t}\cos kx,\,v(x,t)=v_*+\xi_3 e^{\lambda t}\cos kx
\end{equation}
with $0<\xi_1,\xi_2,\xi_3\ll 1.$ The parameter $k$ is the wavenumber of the eigenfunction, and $\lambda$ is the eigenvalue determining the temporal growth of the corresponding $k^{\mathrm{th}}$ mode. We obtain the linearized system as 
\begin{equation}\label{eq:spatio_linear}
       \mathbf{Z}_t = J_{E_*}\mathbf{Z} + \mathcal{D}\Delta \mathbf{Z}
\end{equation}
where $\mathbf{Z}\equiv (z_1,z_2,z_3),$ $\mathcal{D}\equiv\mathrm{diag}(1,1,D).$
For the non-trivial solution of the above system \eqref{eq:spatio_linear}, the eigenvalues $\lambda$ are determined by
the roots of the characteristic polynomial $\det (\lambda I -J_{E_*} + \mathcal{D}k^2)=0,$ which is written explicitly as}
\begin{equation}
  \lambda^3+p_2(k^2)\lambda^2+p_1(k^2)\lambda+p_0(k^2)=0,
\end{equation}
where 
\medskip
\begin{equation}
    \begin{aligned}
         p_2(k^2) = &~~(2+D)k^2-\mathrm{tr}(J_{E_*}),\\
         p_1(k^2) = &~~(1+2D)k^4 -((J_{22}+J_{33})+(J_{11}+J_{33})+D(J_{11}+J_{22}))k^2 + (J_{11}^{[1]}+J_{22}^{[2]}+J_{33}^{[3]}),\\
         p_0(k^2) =&~~ D k^6 - ((J_{11} + J_{22})D + J_{33}) k^4 + ( J_{11}^{[1]} + J_{22}^{[2]} +  J_{33}^{[3]}D) k^2 - \det(J_{E_*}), 
    \end{aligned}
\end{equation}
and $J_{ij}$ and $J_{ii}^{[i]}$ are the same as obtained during the analysis of the temporal part. We, therefore, obtain the necessary and sufficient conditions for Turing instability as
\medskip

\begin{equation}
    \begin{aligned}
         p_2(0)>0,\,\,p_0(0)>0,\,\,p_1(0)p_2(0)>p_0(0)\,\,\text{and}\,\,p_0(k^2)<0,\,\,\text{for some}\,\,k.
    \end{aligned}
\end{equation}

Therefore, the Turing instability occurs at a critical wave number, $k=k_T,$ where $p_0(k^2)$ achieves a minimum and $p_0(k_T^2)=0.$ This gives 
\begin{equation}\label{eq:Turing_critical}
    \begin{aligned}
         k_T^2 = &~~  \frac{J_{11}+J_{22}}{3} + \frac{1}{3D} (J_{33} + \sqrt{\Lambda })
         %\frac{(j_{11}+j_{22})D+j_{33} + \sqrt{\Lambda }}{3D}
     \end{aligned}
\end{equation}

where $$\Lambda = \left(J_{11}^2+J_{22}^2-J_{11}J_{22}+3J_{12}J_{21}\right)D^2 + D\left(3J_{13}J_{31}+3J_{23}J_{32}-J_{11}J_{33}-J_{22}J_{33}\right) + J_{33}^2,$$
where $J_{ij}$ are the elements of the Jacobian matrix, cf.~Eq.~(\ref{eq:Jacobi}). 
Because of the complexity of the expression and the large number of parameters involved, the critical wave number $k_T$ corresponding to the Turing bifurcation has to be computed numerically. For a feasible $k_T$, the model describes the formation of spatial patterns, as is shown below.

\subsection{Impact of diffusion on Oxygen Minimum Zone}\label{sec:OMZ-Turing}

We now look into the spatio-temporal dynamics of the system (\ref{eq:spatio-temp_local}) with the initial conditions (\ref{IC}). Our goal is to reveal typical dynamical regimes (in particular, pattern formation scenarios, if any) for parameters inside and outside of the Turing domain. \textcolor{black}{In a pure Turing domain all the Turing instability conditions as discussed above holds. However, in a Turing-Hopf domain, the homogeneous steady state is unstable under both temporal and spatio-temporal perturbations. } We fix the parameter values as in \eqref{parameters} and let $\mu_1=0,\,\mu_2=0.41$ and consider different values of the diffusivity ratio $D$.  

Fig.~\ref{fig:Turing_1} shows typical patterns in the distribution of oxygen for different diffusivity rates. Here Fig.~\ref{fig:Turing_1}a,b and Fig.~\ref{fig:Turing_1}c are obtained, respectively, for parameters outside and inside the Turing instability domain. 
We notice that, in all three cases, the evolution of the initial conditions soon leads to the formation of a patch where the oxygen concentration is much lower than its steady state value. We interpret this dynamics as the formation of an OMZ. Further evolution of the emerging OMZ can be significantly different depending on $D$. 
When the diffusivity ratio is small, i.e.~$D<1,$ the OMZ created at the early stage grows with time and eventually spreads over the entire domain. Interestingly, the growing OMZ has a fine structure. A closer look at the dynamics shown in Fig.~\ref{fig:Turing_1}a reveals that, at any time $t$ during the transient stage of the OMZ expansion, it consists of three or four subdomains with very low oxygen level separated by narrow spatial intervals where the oxygen level is larger than its steady state value $c_*$. This fine structure disappears after the expanding OMZ hits the domain boundaries; at a later time the oxygen level is low over the entire domain, which can be interpreted as the global anoxia.  

\begin{figure}[!b]
    \centering
  \subfigure[$D=0.8$]{\includegraphics[scale=0.35]{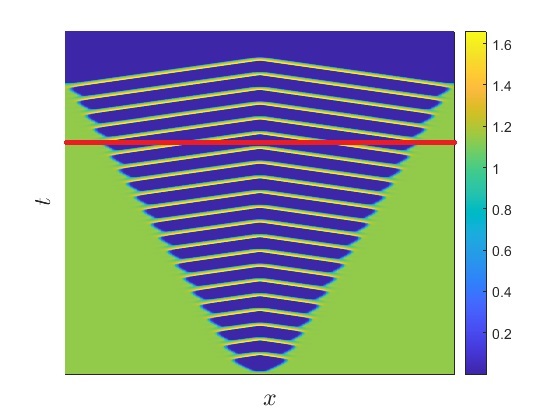}}
    \subfigure[$D=1$]{\includegraphics[scale=0.35]{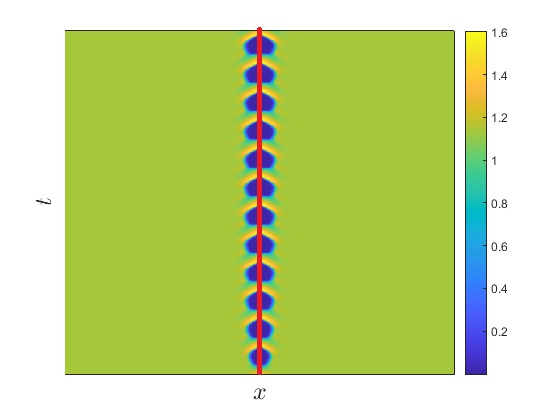}}
     \subfigure[$D=5$]{\includegraphics[scale=0.35]{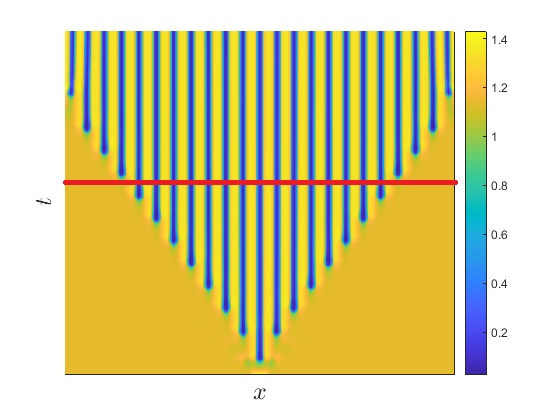}}
     \caption{Transition of spatio-temporal dynamics of oxygen from non-Turing (panels (a) and (b)) to Turing (panel (c)) pattern formation for $\mu_1=0,\, \mu_2=0.41,\,\varepsilon=1$ and different values of $D$. All other parameters are given in \eqref{parameters}; the corresponding steady state value $c_*\approx 1.2$. The auxiliary red lines help to reveal the properties the oxygen distribution of oxygen at a given moment of time (as in panels (a) and (c)) or at a given location in space (as in panel (b)); see details in the text.}    \label{fig:Turing_1}
\end{figure}

An increase in the diffusivity ratio to $D=1$ results in a qualitative change in the dynamics; see Fig.~\ref{fig:Turing_1}(b). In this case, the OMZ formed at an early stage of the system dynamics show almost no spatial growth remaining localised around the centre of the domain. At any spatial position inside the OMZ, the oxygen level distinctly oscillates with time between a very low level (approximately $0.1c_*$) to a high level (of about $1.3c_*$). 

A further increase in $D$ leads to another qualitative change in the dynamical pattern. Figure \ref{fig:Turing_1}(c) shows the results obtained for $D=5$. In this case, the system satisfies the Turing instability condition with the critical wavenumber $k_T^2 = 0.1095.$ The evolution of the initial conditions leads to the formation, inside a certain subdomain, of a periodic spatial pattern where low oxygen patches alternate with high oxygen patches. The subdomain containing this periodic structure grows with time and eventually occupies the whole domain, so that at a large time the periodic spatial distribution becomes stationary.  

We note here that the pattern of the OMZ formation and spread is relatively robust to the initial conditions. For instance, if at $t=0$ the spatial distribution of oxygen is perturbed along with that of phytoplankton, the emerging patterns are similar to the ones shown in Fig.~\ref{fig:Turing_1}. One example is shown in Fig.~\ref{fig:OMZ_area}. In this case, the initial conditions \eqref{IC} are slightly modified, so that, at the center of the domain both $c(x,0)$ and $u(x,0)$ are less than their steady state values. It is readily seen that the top and the bottom of Fig.~\ref{fig:OMZ_area} show very similar patterns, with the only difference that the spatial size of the emerging OMZ becomes larger for the modified initial conditions.  

\begin{figure}[ht!]
    \centering
    \includegraphics[scale=0.4]{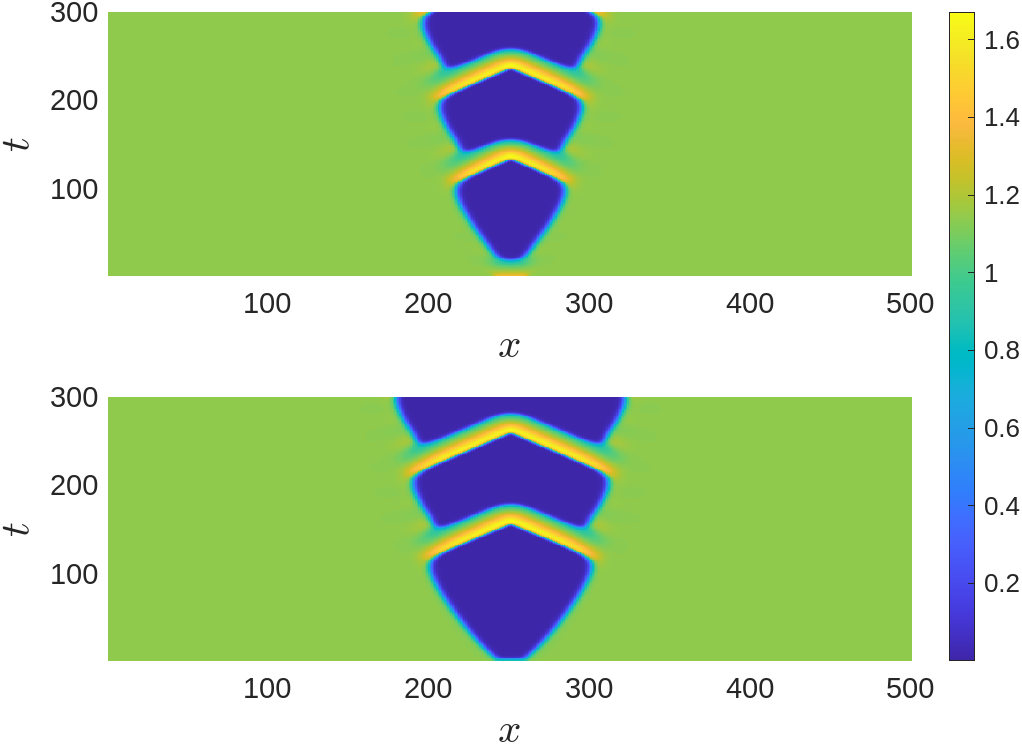}
    \caption{(Top) Zoomed plot of Fig.\ref{fig:Turing_1}(a), which is obtained using the initial condition \eqref{IC}. (Bottom) The initial conditions are of the form \eqref{IC} but with $c(x,0) = c_*-0.5$ and $u(x,0) = u_*-0.2$ for $|x-\frac{L}{2}|<10$. All the parameters are same as in Fig.~\ref{fig:Turing_1}(a). The modified initial conditions therefore result in the formation of the OMZ of a larger size.}
    \label{fig:OMZ_area}
\end{figure}

\subsection{Impact of timescale separation on Oxygen Minimum Zone}

From the mathematical analysis in the subsection \ref{Sec:Turing_instability}, we obtain that the critical wavenumber $k_T$ for Turing instability depends on the timescale separation $\varepsilon$, as the elements $J_{31}$, $J_{32}$ and $J_{33}$ of the Jacobian matrix in Eq.~\ref{eq:Turing_critical} depend on $\varepsilon$. Thus, one can expect that the boundaries of the parameter ranges where the Turing and the Turing-Hopf instability occur (resulting in pattern formation) can shift with a decrease in $\varepsilon$, i.e.~with an increase in the timescale separation. Numerical simulations confirm that this is indeed the case. Having fixed $\mu_1=0,\,\mu_2=0.41$, $D=5$ and other parameters as in \eqref{parameters} and only varying $\varepsilon$, we obtain stationary patterns in both the Turing domain and in the Turing-Hopf domain for $\varepsilon\ge 0.18.$ A typical pattern is shown in Fig.~\ref{fig:Turing_1}(c). With a decrease in $\varepsilon$, the emerging stationary pattern has a similar nature of alternating patches of high and low oxygen level as for $\varepsilon=1$ (cf.~Fig.~\ref{fig:Turing_1}c) but the size of the patches becomes larger; e.g.~see Fig.~\ref{fig:pattern_bif}(a). Also, the emergence of the stationary periodic pattern is preceded by rather long transient dynamics when the oxygen concentration and the plankton densities exhibit irregular oscillations (see Fig.~\ref{fig:pattern_bif}(d)). 

With a further decrease in $\varepsilon$, the dynamics becomes qualitatively different. The emerging pattern is not spatially periodic any more; see Fig.~\ref{fig:pattern_bif}(b). Apart from the large OMZ formed around the centre of the domain at the early stage of system's dynamics, there are two large OMZs at the sides of the domain. At a later time, these patches of low oxygen level break to a number of smaller patches of variable size. The dynamics is not becoming stationary at any time as the patches keep changing their size (and some of them also their location). The dependence of spatially average densities is distinctly irregular (see Fig.~\ref{fig:pattern_bif}(e)) suggesting chaotic dynamics. This kind of dynamic pattern is observed for $0.07<\varepsilon<0.18$. 

With a further decrease in $\varepsilon$ (below $\varepsilon=0.07$), the system's dynamics undergo a regime shift. For $\varepsilon<0.07$, the transient apparently chaotic dynamics only last for a finite time. 
After a sufficiently long time, the system experiences a catastrophic change when over a short transition time the oxygen concentration fast drops to a very small value (and eventually to zero) over the entire spatial domain. An example of such regime shift is shown in Fig.~\ref{fig:pattern_bif}(c,f) (obtained for $\varepsilon=0.06$). The entire area becomes a dead zone (with low or no oxygen), which can be interpreted as the global anoxia. Along with oxygen, the phyto- and zooplankton densities go to zero as well (after several irregular oscillations of increasing amplitude, cf.~Fig.~\ref{fig:pattern_bif}(f)), obviously signifying their extinction. 

\begin{figure}[ht!]
    \centering
    \subfigure[$\varepsilon=0.18$]{\includegraphics[scale=0.35]{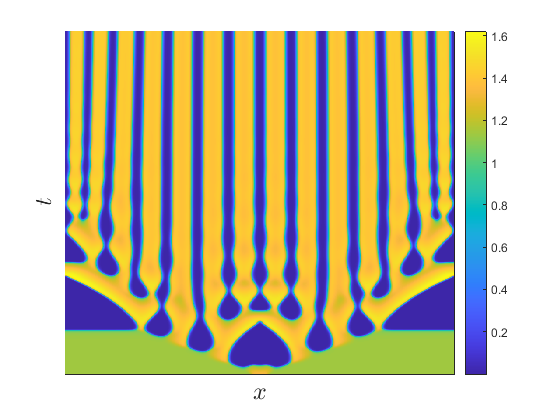}}
    \subfigure[$\varepsilon=0.1$]{\includegraphics[scale=0.35]{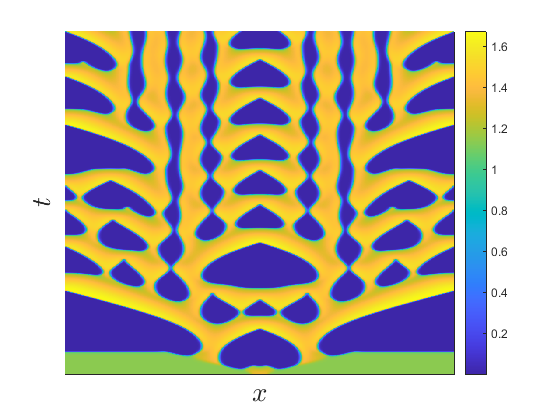}}
    \subfigure[$\varepsilon=0.06$]{\includegraphics[scale=0.35]{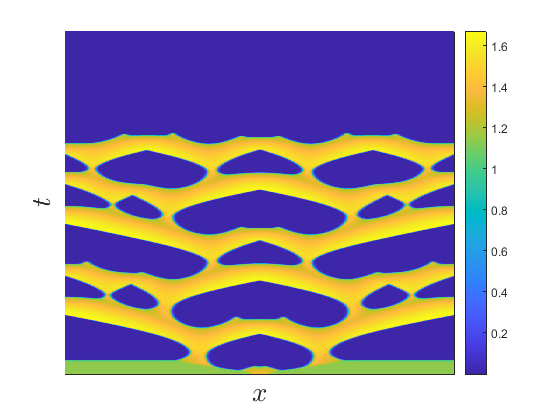}}
    \subfigure[$\varepsilon=0.18$]{\includegraphics[scale=0.35]{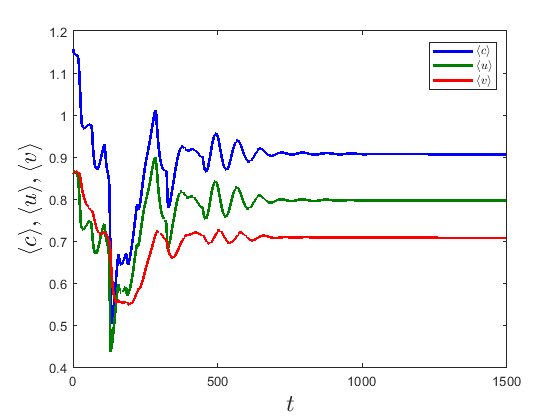}}
    \subfigure[$\varepsilon=0.1$]{\includegraphics[scale=0.35]{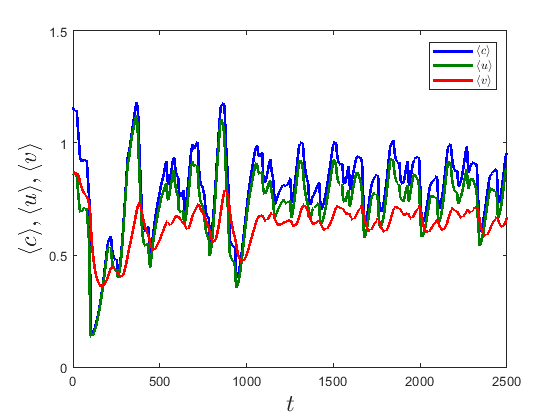}}
    \subfigure[$\varepsilon=0.06$]{\includegraphics[scale=0.35]{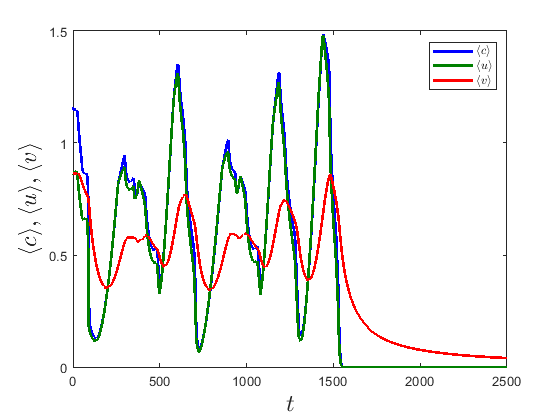}}
    \caption{(a,b,c) Spatial distribution of oxygen;  (d,e,f) spatial average density of oxygen, phytoplankton, and zooplankton for the parameter values \eqref{parameters} with $\mu_1=0,\,\mu_2=0.41,$ and $D=5.$} 
    \label{fig:pattern_bif}
\end{figure}

For both $\mu_1>0$ and $\mu_2>0$, the system's dynamics becomes different and exhibits a somewhat greater variety of dynamical regimes. As one example, Figs.~\ref{fig:pattern_bif_mu1_mu2}(a,e) shows the spatiotemporal dynamics for $\mu_1=0.24,\,\mu_2=0.1575$, $D=\varepsilon=1$ and other parameters the same as in Fig.~\ref{fig:Turing_1}(b,e). It is readily seen that, in this particular case, the evolution of the initial condition does not lead to formation of OMZ. It only leads to small fluctuations in the oxygen level and plankton densities around the location of the initial perturbation, with the spatial distribution being uniform in the rest of the domain.

A decrease in $\varepsilon$ leads to the emergence of the OMZ. It first appears at the position of the initial perturbation (i.e.~near the centre of the domain); see the bottom of Figs.~\ref{fig:pattern_bif_mu1_mu2}(b,c). At a later time, it breaks to several patches that fast spread over the entire domain. The spatiotemporal dynamics is apparently chaotic for $\varepsilon=0.5$ but becomes more regular for $\varepsilon=0.25$, cf.~Figs.~\ref{fig:pattern_bif_mu1_mu2}(f,g). 

A further decrease in $\varepsilon$ below a certain critical value results in a regime shift, e.g.~see Figs.~\ref{fig:pattern_bif_mu1_mu2}(d,h) obtained for $\varepsilon=0.2$. In this case, a large OMZ is formed at an early stage of system's dynamics (see the bottom of Fig.~\ref{fig:pattern_bif_mu1_mu2}(d)). However, after a relatively short time the oxygen concentration fast drops to zero over the entire domain: the global anoxia occurs accompanied by the plankton extinction. 

%%%%%%%%%%%%%%%%%%%%%%%%
\begin{figure}[ht!]
    \centering
    \subfigure[$\varepsilon=1$]{\includegraphics[scale=0.25]{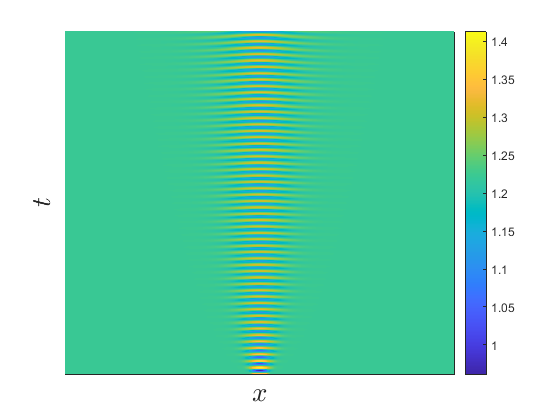}}
    \subfigure[$\varepsilon=0.5$]{\includegraphics[scale=0.25]{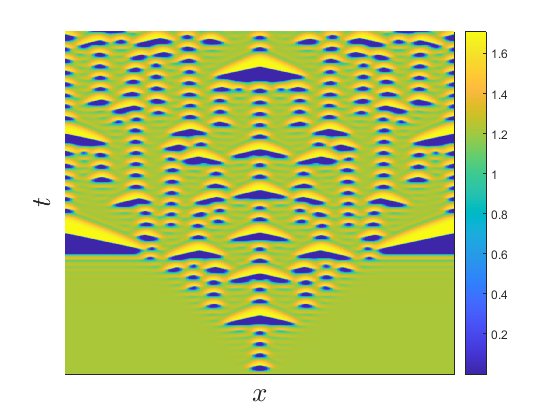}}
    \subfigure[$\varepsilon=0.25$]{\includegraphics[scale=0.25]{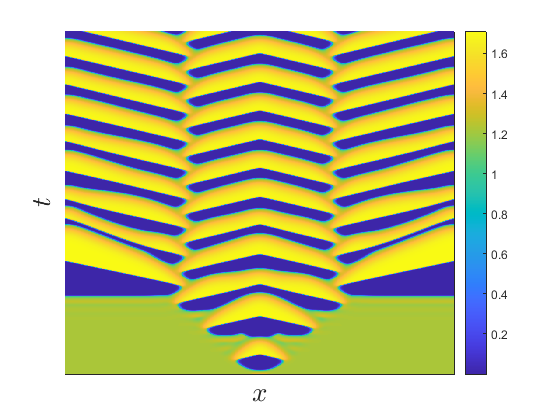}}
     \subfigure[$\varepsilon=0.2$]{\includegraphics[scale=0.25]{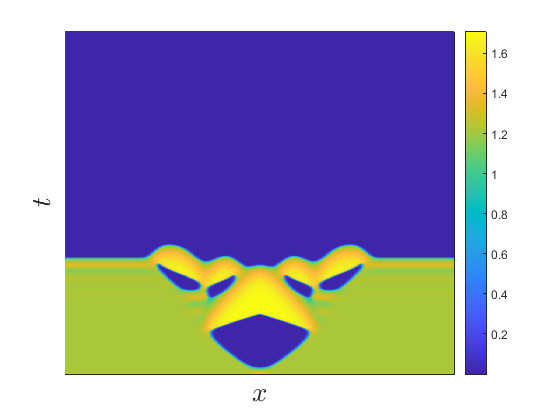}}
    
    \subfigure[$\varepsilon=1$]{\includegraphics[scale=0.25]{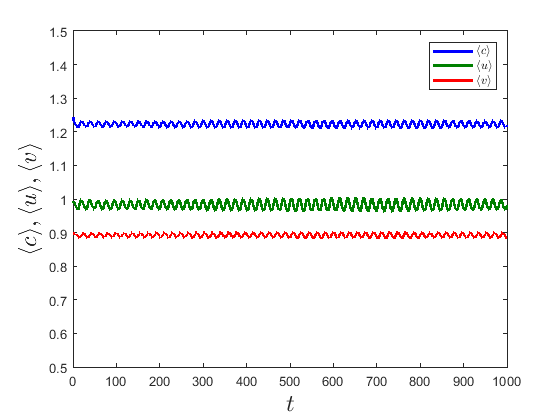}}
    \subfigure[$\varepsilon=0.5$]{\includegraphics[scale=0.25]{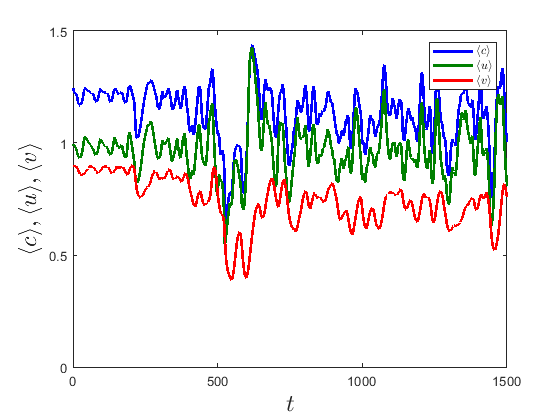}}
    \subfigure[$\varepsilon=0.25$]{\includegraphics[scale=0.25]{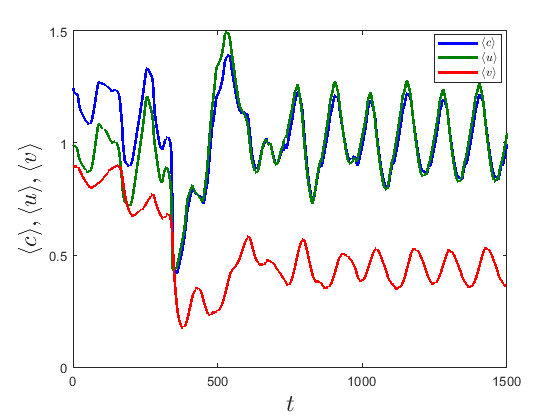}}
    \subfigure[$\varepsilon=0.2$]{\includegraphics[scale=0.25]{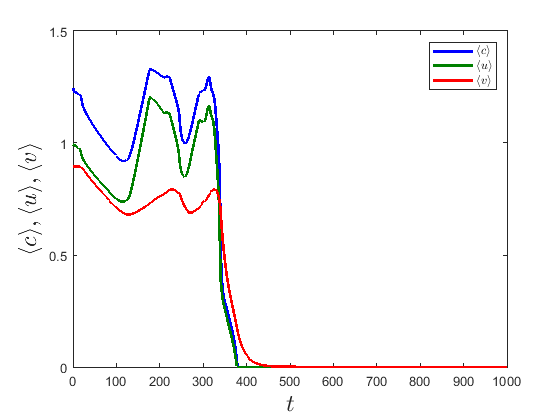}}
    \caption{(a,b,c,d) Transition of spatial distribution of oxygen and (e,f,g,h) spatial average density of oxygen, phytoplankton, and zooplankton for the parameter values \eqref{parameters} and $\mu_1=0.24,\,\mu_2=0.1575,\,D=1$ and different values of $\varepsilon.$ }
    \label{fig:pattern_bif_mu1_mu2}
\end{figure}

\section{Conclusion}\label{sec:conclus}

Over the last few decades, there has been growing evidence of a decline of the dissolved oxygen concentration in the ocean \cite{Breitburg18,Schmidtko17}. This has not only been recognized as a catastrophic threat to the marine ecosystems \cite{Wishner} but also as a potential threat to mankind \cite{Martin17} and to terrestrial ecosystems, as marine phytoplankton contributes about 70\% to the total atmospheric oxygen. Any significant decline in the global phytoplankton abundance and/or a decrease in the oxygen production rate in phytoplankton photosynthesis will inevitably lead to a decline in the global stock of the atmospheric oxygen \cite{Petrovskii17,Petrovskii21}. Thus, marine ecosystems, phytoplankton in particular, play a crucial role in maintaining the habitable Earth \cite{Sudakow22}. 

In spite of the apparent importance of the above issues, mathematical models addressing the change in the oxygen concentration as a component of the coupled phytoplankton-oxygen dynamics are rare in the literature. As one exception, a generic three-component oxygen-phyto-zooplankton model was developed in \cite{sekerci15a} (and further investigated in \cite{Petrovskii17,Sekerci18,Sekerci18b}). It has been shown that the formation of areas with a low oxygen concentration (i.e.~OMZs) is in fact an inherent property of the self-organised plankton-oxygen spatiotemporal dynamics, but it can be exacerbated by the effect of global warming, potentially leading to global anoxia. 

The model developed in \cite{sekerci15a,Petrovskii17}, however, missed several important features of the marine ecosystem's dynamics, hence making the prediction of emerging global anoxia somewhat questionable.  
In this paper, we have considered a nontrivial extension of the original model that includes factors such as zooplankton inherent competition, cannibalism and/or the effect of zooplankton's consumers from upper trophic layers, e.g.~fish. Another important factor is the existence of different timescales for phyto- and zooplankton growth, as the latter is usually much slower than the former. 

The properties of the extended model have been analysed in much detail using a combination of analytical and numerical tools. 
We first consider a non-spatial version of the model described by a system of three nonlinear ODEs (for oxygen, phytoplankton, and zooplankton, respectively) to study the variation in oxygen level and plankton densities over time. In addition to the results earlier obtained in \cite{sekerci15a,Sekerci15b}, we have shown that the number of coexisting steady states depends both on the linear mortality rate ($\mu_1$) and the rate of zooplankton intraspecific competition/consumption (quantified by coefficient $\mu_2$). For $\mu_1=0,$ there exists two feasible coexisting steady states where the lower oxygen level is always unstable, and the higher state changes its stability with increasing $\mu_2.$ Whereas for $\mu_1\ne0,$ we obtain a unique feasible steady state which changes its stability from unstable to stable for a strong intraspecific competition. Along with this, the extinction state is always stable. We also found that an increase in the rates of zooplankton linear mortality ($\mu_1$) and nonlinear mortality ($\mu_2$) leads to an increase in the oxygen abundance.

In order to better understand the relative importance of the linear and nonlinear mortality and their effect on the temporal dynamics of the system, we have considered three cases: (a)  $\mu_1\ne 0,\mu_2=0,$ (b) $\mu_1=0,\mu_2\ne0,$ and (c) $\mu_1\ne 0,\mu_2\ne 0$. Because of the complexity of the system, this has mostly been done through numerical simulations. For case (a), the unique steady state is stable for higher values of $\mu_1,$ and it loses its stability through supercritical Hopf bifurcation. Small stable cycles originate, increasing its size in a small interval of $\mu_2.$ Beyond that, the system cannot further withstand an increase in the amplitude of the cycle leading to complete collapse (see Fig.~\ref{fig:local_dynamics_temporal}a). However, for case (b), the Hopf bifurcation is subcritical, and the  system converges to stable steady state for higher values of $\mu_2.$ In this case, an unstable cycle is formed (see Fig.~\ref{fig:local_dynamics_temporal}b). Case (c) is a combination of the above two cases. Here, in a very narrow domain, the system exhibits tri-stability, with a stable steady state, a stable cycle, and the extinction state. \textcolor{black}{The two cycles appear through saddle-node bifurcation of limit cycle, and the disappearance of the unstable cycle is through subcritical Hopf, and that of the outer stable cycle is through heteroclinic bifurcation (see Fig.~\ref{fig:local_dynamics_temporal}c). }

Having analysed the effect of different timescales (cf.~``slow-fast system''), we obtained the critical manifold of the slow subsystem. The extremum (the fold point) of the critical manifold acts as an extinction threshold: if the system is pushed beyond the fold point (e.g.~by the choice of the initial conditions), the dynamics will eventually lead to plankton extinction and oxygen depletion, although the extinction/anoxia can be preceded by a long period of oscillations (cf.~Fig.~\ref{fig:trajectory_slow_fast}). Note that a decrease in the nonlinear mortality rate $\mu_2$ has a similar effect on the system's persistence. A decrease in $\mu_2$ below a certain critical value first destabilises the coexistence steady state resulting in oscillatory dynamics (see Fig.~\ref{fig:bifurcation_canard_explosion}). A further decrease (below another critical value) leads to the canard explosion. However, no relaxation oscillations emerge (see Theorem 4.1); instead, the system's trajectory goes to the origin, which obviously corresponds to the extinctions and anoxia. 

We then considered the spatially explicit system to study the distribution and spatiotemporal dynamics of oxygen and plankton. 
The spatially explicit model consists of three reaction-diffusion equations where the diffusion terms account for the effect of lateral turbulent mixing for the dissolved oxygen and phytoplankton and for the combined effect of turbulence and self-motion for zooplankton. Note that, because of the latter, zooplankton diffusivity can be expected to differ from that of phytoplankton and oxygen. 
Moreover, it can differ significantly. 
The interplay between the ordinary Fickian diffusion (in our case ``biodiffusion'' resulting from zooplankton random movement) and the turbulent mixing is known to be highly nonlinear \cite{Monin71,Monin75}. The ordinary diffusion, although itself often being orders of magnitude less intensive than the turbulent mixing, accelerates the turbulent diffusion significantly \cite{Monin71,Monin75}. In turn, for the diffusivity ratio being greater than one, the model can exhibit pattern formation due to the Turing instability; see Sections \ref{Sec:Turing_instability} and \ref{sec:OMZ-Turing}. 

Due to its mathematical complexity, the spatially explicit reaction-diffusion model is not analytically tractable. Therefore, we have investigated its properties through extensive numerical simulations, with a special attention to regimes that result in the formation of patterns containing areas with low oxygen level and/or regimes resulting in global oxygen depletion. Using the initial condition as a localised perturbation of the spatially uniform steady state, we have obtained that the system dynamics typically lead to the formation of strongly heterogeneous spatial distribution that includes one or several areas (patches) with a very low oxygen level, which we interpret as the formation of OMZ. Interestingly, the patterns emerge both in and outside of the Turing domain and hence, for different parameter values (e.g.~the diffusivity ratio being larger or smaller than one) can be attributed to different dynamical mechanisms, i.e.~Turing or non-Turing. Except for some rare cases (cf.~Fig.~\ref{fig:Turing_1}b), the OMZ formed at an early stage of the system dynamics fast spreads over the entire domain, often generating multiple patches, e.g.~see Figs.~\ref{fig:Turing_1}(a,c), \ref{fig:pattern_bif}(a,b,c) and
\ref{fig:pattern_bif_mu1_mu2}(b,c,d), the size and number of the emerging smaller OMZs varying with the parameter values.

The spread of the emerging pattern (a mixture of patches with high and low oxygen level) can lead to a different outcome. It can result in a self-sustained pattern, which, in the large time limit, can be stationary (cf.~Figs.~\ref{fig:Turing_1}c and \ref{fig:pattern_bif}a) or dynamic (Fig.~\ref{fig:pattern_bif}b). Alternatively, it may eventually lead to an unsustainable pattern -- a regime shift -- when, after a certain time, the oxygen concentration fast drops to very small values over the entire domain (cf.~Fig.~\ref{fig:Turing_1}a, \ref{fig:pattern_bif}c and \ref{fig:pattern_bif_mu1_mu2}d). Arguably, this may be interpreted as the onset of the global anoxia. 

Note that there is a subtle interplay between the zooplankton linear mortality rate $\mu_1$ and the difference in the timescales. In the special case of the same timescales ($\varepsilon=1$), the effect of a non-zero zooplankton linear mortality makes the system somewhat more sustainable: while an initial perturbation leads to the formation of deoxygenated patch (OMZ) at the center of the domain in case $\mu_1=0$ (Fig.~\ref{fig:Turing_1}b), it only leads to small fluctuations in the oxygen level in case $\mu_1>0$ (Fig.~\ref{fig:pattern_bif_mu1_mu2}a). 
However, for $\varepsilon<1$ and with a further increase in the timescale separation (i.e.~for a smaller $\varepsilon$), the effect of zooplankton mortality becomes rather opposite making the system less sustainable. For instance, for $\mu_1>0$ the global anoxia occurs already for $\varepsilon=0.2$ (see Fig.~\ref{fig:pattern_bif_mu1_mu2}d) but for $\mu_1=0$ the dynamics remains sustainable (i.e.~no global anoxia) for $\varepsilon=0.1$. 

For a fixed value of $\mu_1$, an increase in the timescale separation alone makes the dynamics less sustainable. A decrease in $\varepsilon$
tends to lead to a larger size of the initially formed OMZ; e.g.~see Fig.~\ref{fig:pattern_bif_mu1_mu2}(b,c,d), eventually resulting in global anoxia and extinctions when $\varepsilon$ becomes sufficiently small. That happens both for $\mu_1=0$ and $\mu_1>0$, cf.~Figs.~\ref{fig:pattern_bif}(c) and \ref{fig:pattern_bif_mu1_mu2}(d), although the succession of spatial patterns preceding the onset of anoxia is different between the two cases.

Apparently, our study leaves open questions. Firstly, recall that our model is conceptual; it only takes into account the interaction between oxygen and plankton but not with other components of the complicated marine food web. It has been shown in \cite{Petrovskii17} that, in case of a trophic chain, the effect of higher trophic levels only makes the regime shift -- the catastrophe of oxygen depletion and plankton extinction -- more likely as the three-component model \eqref{eq:temporal} provides an upper bound for a longer trophic chain. An open question however remains as to how the dynamics may change in case of a web rather than chain, for instance to account for effect of bacteria or detritus. 
Secondly, the description of turbulent mixing as the turbulent diffusion is somewhat simplistic; in particular, it completely disregards the fact that the turbulent mixing is multiscale and nonlocal\cite{Monin71,Monin75}. Although the model with the turbulent diffusion is arguably a sensible first step, a more advanced approach should involve a more realistic description of turbulence. 
These issues will become a focus of future work. 

\section*{Data Availability}

This paper has no data. 

\section*{Conflict of Interest}

The authors do not report any conflict of interest. 
\\


\newpage 

\begin{thebibliography}{99}

% Bibliography in order of appearance in text

\bibitem{Heinze21}{Heinze, C., Blenckner, T., Martins, H., Rusiecka, D., Döscher, R., Gehlen, M., Gruber, N., Holland, E., Hov, Ø., Joos, F., Matthews, J., Rødven, R., Wilson, S. (2021). The quiet crossing of ocean tipping points. Proceedings of the National Academy of Sciences of the United States of America, 118(9), e2008478118.}

\bibitem{Watson17}{Watson, A. J., Lenton, T. M., Mills, B. (2017). Ocean deoxygenation, the global phosphorus cycle and the possibility of human-caused large-scale ocean anoxia. Philosophical transactions. Series A, Mathematical, physical, and engineering sciences, 375(2102), 20160318.}

\bibitem{Wignall96}{Wignall, P. B., Twitchett, R. J. (1996). Oceanic Anoxia and the End Permian Mass Extinction. Science (New York, N.Y.), 272(5265), 1155–1158.}

\bibitem{Petsch}{Petsch, S. T. (2003). The global oxygen cycle. Treatise on geochemistry, 8, 682.}

\bibitem{Berner}{Berner RA. Atmospheric oxygen over Phanerozoic time. Proc Natl Acad Sci U S A. 1999 Sep 28;96(20):10955-7.}


\bibitem{Schmidtko17}{Schmidtko, S., Stramma, L., Visbeck, M. (2017). Decline in global oceanic oxygen content during the past five decades. Nature, 542(7641), 335–339.}


\bibitem{Breitburg18}{Breitburg D, Levin LA, Oschlies A, Grégoire M, Chavez FP, Conley DJ, Garçon V, Gilbert D, Gutiérrez D, Isensee K, Jacinto GS, Limburg KE, Montes I, Naqvi SWA, Pitcher GC, Rabalais NN, Roman MR, Rose KA, Seibel BA, Telszewski M, Yasuhara M, Zhang J. Declining oxygen in the global ocean and coastal waters. Science. 2018 Jan 5;359(6371):eaam7240.}

\bibitem{Keeling}{Keeling, R. F., Körtzinger, A., Gruber, N. (2010). Ocean deoxygenation in a warming world. Annu. Rev. Mar. Sci, 2(1), 199-229.}

\bibitem{Oschlies18}{Oschlies, A., Brandt, P., Stramma, L., Schmidtko, S. (2018). Drivers and mechanisms of ocean deoxygenation. Nature Geoscience, 11(7), 467-473.}

\bibitem{Oschlies17}{Oschlies, A., Duteil, O., Getzlaff, J., Koeve, W., Landolfi, A., Schmidtko, S. (2017). Patterns of deoxygenation: sensitivity to natural and anthropogenic drivers. Philosophical transactions. Series A, Mathematical, physical, and engineering sciences, 375(2102), 20160325.}

\bibitem{Helm11}
K. P. Helm, N. L. Bindoff, J. A. Church, Observed decreases in
oxygen content of the global ocean. Geophys. Res. Lett. 38,
L23602 (2011).


\bibitem{Matear00}
Matear RJ, Hirst AC, McNeil BI. 2000. Changes in dissolved oxygen in the Southern Ocean with climate
change. Geochem. Geophys. Geosystems 1, 2000GC000086

\bibitem{Bopp02}
Bopp L, Le Quere C, Heimann M, Manning AC, Monfray P. 2002. Climate-induced oxygen fluxes: implications
for the contemporary carbon budget. Glob. Biogeochem. Cycles 16, 1022 doi:10.1029/2001GB001445

\bibitem{Jones77} 
Jones, R.I. The importance of temperature conditioning to the respiration of natural phytoplankton
communities. Br. Phycol. J. 1977, 12, 277-285. 

\bibitem{Li84}
Li, W.; Smith, J.; Platt, T. Temperature response of photosynthetic capacity and carboxylase activity in arctic marine phytoplankton. Mar. Ecol. Prog. Ser. 1984, 17, 237-243.

\bibitem{Robinson00}
Robinson, C. Plankton gross production and respiration in the shallow water hydrothermal systems of Milos,
Aegean Sea. J. Plankton Res. 2000, 22, 887-906.

\bibitem{Hancke04}
Hancke, K.; Glud, R.N. Temperature effects on respiration and photosynthesis in three diatom-dominated
benthic communities. Aquat. Microb. Ecol. 2004, 37, 265-281.

\bibitem{Petrovskii17}{Petrovskii SV, Sekerci Y, Venturino E. (2017) Regime shifts and ecological catastrophes in a model of plankton‐
	oxygen dynamics under the climate change. J. Theor. Biol. 424, 91‐109. }

\bibitem{sekerci15a} Sekerci, Y., Petrovskii, S.V. (2015) Mathematical modelling of plankton-oxygen dynamics under the climate change. {\it{Bull. Math. Biol.}} {\bf{77}}, 2325-2353.

\bibitem{Sekerci18} Sekerci, Y., Petrovskii, S.V. (2018) Global warming can lead to depletion of oxygen by disrupting phytoplankton photosynthesis: a mathematical modelling approach. Geosciences 8, 201.

\bibitem{Erwin14}
Erwin D. 2014. Temporal acuity and the rate and dynamics of mass extinctions. Proc Natl Acad Sci 111:3203-3204.

\bibitem{Song14}
Song H, Wignall PB, Chu D, et al. 2014. Anoxia/high temperature double whammy during the Permian-Triassic marine crisis and its aftermath. Sci Rep 4(1):1–7.


\bibitem{Sudakow22} I Sudakow, C Myers, S Petrovskii, CD Sumrall, J Witts (2022) Knowledge gaps and missing links in understanding mass extinctions: Can mathematical modeling help?
Physics of Life Reviews 41, 22-57. 

\bibitem{Beltrami}{E. Beltrami, Unusual algal blooms as excitable systems: the case of brown-tides, Environ. Modeling
	Assessment 1 (1996) 19–24.}

\bibitem{Beltrami_Carroll}{F.A. Ascioti, E. Beltrami, T.O. Carroll, C. Wirick, Is there chaos in plankton dynamics? J. Plankton
	Res. 15 (1993) 603–617.}

\bibitem{Brindley96}{A.M. Edwards, J. Brindley, Oscillatory behaviour in a three-component plankton model, Dyn. Stability
	Systems 11 (1996) 347–370.}

\bibitem{Steele74}{J.H. Steele, The Structure of Marine Ecosystems. Blackwell Scienti c Publ., Oxford, 1974.}

\bibitem{Steele92}{J.H. Steele, E.W. Henderson, A simple model for plankton patchiness, J. Plankton Res. 14 (1992)
	1397–1403.}

\bibitem{Cushing}{D.H. Cushing (1975), Marine Ecology and Fisheries, Cambridge University Press, Cambridge}

\bibitem{Truscott94a}{J.E. Truscott, J. Brindley, Ocean plankton populations as excitable media, Bull. Math. Biol. 56 (1994a)
	981–998.}

\bibitem{Truscott94b}{J.E. Truscott, J. Brindley, Equilibria, stability and excitability in a general class of plankton population
	models, Philos. Trans. Roy. Soc. London A 347 (1994b) 703–718.}

 \bibitem{Steel78}{Steele, J.H., ed. (1978). Spatial Pattern in Plankton Communities, NATO Conference Series: IV, Marine Sciences, Vol. 3, Plenum Press, New York.}


\bibitem{Medvinsky02} {Medvinsky, A., Petrovskii, S., Tikhonova., I., Malchow., H.,
	Li, B.L.: Spatiotemporal complexity of plankton and fish dynamics. {\it
		SIAM Review} {\bf 44}(3), 311–370 (2002)}

\bibitem{Malchow00}{Malchow, H., Radtke, B., Kallache, M., Medvinsky, A. B., Tikhonov, D. A., Petrovskii, S. V. (2000). Spatio-temporal pattern formation in coupled models of plankton dynamics and fish school motion. Nonlinear Analysis: Real World Applications, 1(1), 53-67.}
	

\bibitem{Scheffer00}{Scheffer, M., Rinaldi, S., Kuznetsov, Y.A.: Effects of fish
	on plankton dynamics: A theoretical analysis, {\it Canadian Journal of
		Fisheries and Aquatic Sciences} {\bf 57}(6), 1208-1219 (2000)}

\bibitem{Scheffer97}{M. Scheffer, S. Rinaldi, Yu.A. Kuznetsov, and E.H. van Nes (1997), Seasonal dynamics of
	Daphnia and algae explained as a periodically forced predator-prey system, OIKOS, 80, pp.
	519–532.}

\bibitem{Martin03}{Martin, A.P. (2003) Phytoplankton patchiness: the role of lateral stirring and mixing. Progress in oceanography 57 (2), 125-174.}

	
\bibitem{Scheffer91a}{M. Scheffer (1991a), Fish and nutrients interplay determines algal biomass: A minimal model,
	OIKOS, 62, pp. 271–282.}


\bibitem{Petrovskii00}{Petrovskii, S.V., Malchow, H.: Critical phenomena
	in plankton communities: KISS model revisited. {\it Nonlinear Analysis
		Real World Applications} {\bf 1}(1), 37-51 (2000)}

\bibitem{Petrovskii01TPB}{Petrovskii, S.V., Malchow, H.: Wave of chaos: new
	mechanism of pattern formation in spatio-temporal population dynamics.
	{\it Theoretical Population Biology} {\bf 59}(2), 157-174 (2001)}

\bibitem{Petrovskii02b} {Petrovskii, S., Vinogradov, M.E., Morozov, A.: Formation of the patchiness in the plankton horizontal distribution due to biological invasion in a two-species model with account for the Allee effect. {\it Oceanology} {\bf 42}, 363–372 (2002)}

\bibitem{Rinaldi93Eco}{S. Rinaldi and S. Muratori (1993), Conditioned chaos in seasonally perturbed predator-prey
	models, Ecological Modelling, 69, pp. 79–97.}

\bibitem{Charlson87}
Charlson RJ, Lovelock JE, Andreae MO, Warren SG (1987) Oceanic phytoplankton,
atmospheric sulphur, cloud albedo and climate. Nature 326, 655-661.


\bibitem{Sarmiento98}
Sarmiento JL, Hughes TMC, Stouffer RJ, Manabe S. 1998. Simulated response of the ocean carbon cycle to
anthropogenic climate warming. Nature 393:245-49. 


\bibitem{Hays05}{Hays, G. C., Richardson, A. J.,  Robinson, C. (2005). Climate change and marine plankton. Trends in ecology and evolution, 20(6), 337-344.}

\bibitem{Sekerci15b}{Sekerci, Y., Petrovskii, S.V. (2015). Mathematical Modelling of Spatiotemporal Dynamics of Oxygen in a Plankton System. Mathematical Modelling of Natural Phenomena, 10, 96-114.}

\bibitem{Petrovskii21}
Petrovskii, S.V. (2021) Global Warming Can Result in Global Anoxia by Disrupting Phytoplankton Photosynthesis. Encyclopedia of Climate Change: Case Studies of Climate Risk, Action, and Opportunity, Volume 2, Pages 243-249. World Scientific, Singapore. 



\bibitem{Battaglia17}
{Battaglia, G., Joos, F. (2017) Hazards of decreasing marine oxygen: The near-term and millennial-scale benefits of meeting the Paris climate targets. Earth Syst. Dyn. Discuss., doi:10.5194/esd-2017-90}


\bibitem{Martin17}{Martin, D., McKenna, H., Livina, V. (2017)
	The human physiological impact of global deoxygenation. 
	Journal of Physiological Sciences 67(1), 97-106.}

\bibitem{Ito17}
Ito, T.; Minobe, S.; Long, M.C.; Deutsch, C. Upper ocean O2 trends: 1958–2015. Geophys. Res. Lett. 2017, 44,
4214–4223.


\bibitem{Richardson17}
Richardson, K.; Bendtsen, J. Photosynthetic oxygen production in a warmer ocean: The Sargasso Sea as a
case study. Phil. Trans. R. Soc. A 2017, 375, 20160329.

\bibitem{Stramma08}{Stramma, L., Johnson, G. C., Sprintall, J., Mohrholz, V. (2008). Expanding oxygen-minimum zones in the tropical oceans. Science (New York, N.Y.), 320(5876), 655–658.}

\bibitem{Morrison99}{Morrison, J.M., Codispoti, L.A., Smith, S.L., Wishner, K., Flagg, C., Gardner, W.D., Gaurin, S., Naqvi, S.W.A., Manghnani, V., Prosperie, L. and Gundersen, J.S., 1999. The oxygen minimum zone in the Arabian Sea during 1995. Deep Sea Research Part II: Topical Studies in Oceanography, 46(8-9), pp.1903-1931.}

\bibitem{Diaz}{Diaz, R. J., Rosenberg, R. (2008). Spreading dead zones and consequences for marine ecosystems. Science, 321(5891), 926-929.}

\bibitem{Wishner}{Wishner, K. F., Seibel, B. A., Roman, C., Deutsch, C., Outram, D., Shaw, C. T., Birk, M. A., Mislan, K., Adams, T. J., Moore, D., Riley, S. (2018). Ocean deoxygenation and zooplankton: Very small oxygen differences matter. Science advances, 4(12), eaau5180.}

\bibitem{Alhassan23}
Alhassan Y (2023) Standing on a cliff: OMZ growth may indicate the approach of a global anoxia. Submitted. 


\bibitem{Lenton08}{Lenton, T. M., Held, H., Kriegler, E., Hall, J. W., Lucht, W., Rahmstorf, S., Schellnhuber, H. J. (2008). Tipping elements in the Earth's climate system. Proceedings of the National Academy of Sciences of the United States of America, 105(6), 1786–1793.}

\bibitem{Bazykin98} Bazykin AD. 1998. Nonlinear dynamics of interacting populations. Singapore: World Scientific.

\bibitem{Chowdhury23a} Chowdhury PR, Petrovskii S, Volpert V, Banerjee M. (2023)
Attractors and long transients in a spatio-temporal slow–fast Bazykin’s model.
Communications in Nonlinear Science and Numerical Simulation 118, 107014. 

\bibitem{Rinaldi92}
Rinaldi S, Muratori S. 1992. Slow-fast limit cycles in Predator–Prey models. Ecol Model 61:287-308.

\bibitem{Kuehn15} Kuehn, C. (2015). Multiple Time Scale Dynamics. {\it Springer}, New York.

\bibitem{Poggiale20}
Poggiale JC, Aldebert C, Girardot B, Kooi BW. 2020. Analysis of a Predator–Prey model
with specific time scales: a geometrical approach proving the occurrence of
canard solutions. J Math Biol 80:39-60. 

\bibitem{Sadhu21}
Sadhu S (2021) Complex oscillatory patterns near singular Hopf bifurcation in a two-timescale ecosystem. 
Discrete \& Continuous Dynamical Systems B 26, 5251-5279.

\bibitem{Chowdhury23b} Chowdhury PR, Banerjee M, Petrovskii S. (2023)
Coexistence of chaotic and non-chaotic attractors in a three-species slow–fast system. 
Chaos, Solitons and Fractals 167, 113015.

\bibitem{Steele81} Steele, J. H., Henderson, E. W. (1981). A simple plankton model. The American Naturalist, 117(5), 676-691.

\bibitem{Franke99} Franke, Ulrich, Kolumban Hutter, and Klaus Jöhnk. "A physical–biological coupled model for algal dynamics in lakes." Bulletin of Mathematical Biology 61, no. 2 (1999): 239-272.

\bibitem{SteeleHenderson92}{Steele, J. H., Henderson, E. W. (1992). The role of predation in plankton models. Journal of Plankton Research, 14(1), 157-172.}

\bibitem{Franks02}
Franks, P.J.S. (2002). NPZ models of plankton dynamics: their construction, coupling to physics, and applications. J. Oceanogr. 58, 379–387.

%\bibitem{PetrovskiiRSP05}{Petrovskii, S., Malchow, H., Li, B. L. (2005). An exact solution of a diffusive predator–prey system. Proceedings of the Royal Society A: Mathematical, Physical and Engineering Sciences, 461(2056), 1029-1053.}

\bibitem{Fenichel} Fenichel, N. (1979). Geometric singular perturbation theory for ordinary differential equations. { \it Journal of Differential Equations} {\bf 31}, 53–-98.

\bibitem{Krupa01b} Krupa, M., Szmolyan, P. (2001a). Relaxation oscillation and Canard explosion. {\it Journal of Differential Equations} {\bf 174}, 312–-368.

\bibitem{Levin90} Levin, S. A. (1990). Physical and biological scales and the modelling of predator–prey interactions in large marine ecosystems, in Large Marine Ecosystems: Patterns, Processes and Yields, K. Sherman, L. M. Alexander and B. D. Gold (Eds), Washington: AAAS, pp. 179–187.

\bibitem{Okubo80}
Okubo A (1980) Diffusion and ecological problems: mathematical models. Springer, Berlin. 

\bibitem{Monin71}
Monin AS, Yaglom AM (1971) Statistical Fluid Mechanics: Mechanics of Turbulence. Vol.~1. MIT Press,
Cambridge MA. 

\bibitem{Sekerci18b} Sekerci, Y., Petrovskii, S.V. (2018) Pattern formation in a model oxygen-plankton system. Computation 6, 59; doi:10.3390/computation6040059



\bibitem{Monin75} Monin AS, Yaglom AM (1975) 
Statistical Fluid Mechanics: Mechanics of Turbulence. Vol.~2. MIT Press, Cambridge MA. 


%%%%%%%%%%%%%%%%%%%%%%%%%%%%%%%%%%%%%%%%%%%%%%%


\end{thebibliography}
\end{document}